\documentclass[leqno, a4, fleqn, 10pt]{amsart}

\usepackage{lineno}
\usepackage{ulem}
\usepackage{comment}
\usepackage{soul}

\usepackage{xparse,amsmath}

\ExplSyntaxOn
\NewDocumentCommand{\cfracdots}{ }
  {
   \rule{0pt}{1.5\baselineskip}
   \raisebox{.5\baselineskip}{\enspace$\ddots$\enspace}
  }

\NewDocumentCommand{\xcontfrac}{ s O{c} >{\SplitArgument{1}{;}}m }
  { 
   \IfBooleanTF{#1}
     { \cfrac_inline:nn #3 }
     { \cfrac_map:nnn { #2 } #3 }
  }

\cs_new:Npn \cfrac_inline:nn #1 #2
  {
   \IfNoValueTF { #2 }
     {
      \tl_use:N \c_cfrac_message_tl
      \xcontfrac*{;#1}
     }
     {
      \group_begin:
      \cs_set_eq:NN \cfracdots \dots
      [\, \tl_if_empty:nTF { #1 } { 0 } { #1 } ; #2 \,]
      \group_end:
     }
  }

\tl_const:Nn \c_cfrac_lbrace_tl { \if_true:  { \else: } \fi: }
\tl_const:Nn \c_cfrac_rbrace_tl { \if_false: { \else: } \fi: }
\tl_const:Nn \c_cfrac_strut_tl { \vrule width 0pt depth .3\baselineskip }
\tl_new:N \l_cfrac_left_tl
\tl_new:N \l_cfrac_right_tl
\msg_new:nnn { cfrac } { wrong-syntax }
  {
   Wrong~syntax~for~\token_to_str:N \xcontfrac,~
   assuming~0~in~the~integer~part,~on~line~\msg_line_number:.
  }

\cs_new:Npn \cfrac_map:nnn #1 #2 #3
  {
   \tl_clear:N \l_cfrac_left_tl \tl_clear:N \l_cfrac_right_tl
   \IfNoValueTF { #3 }
     { 
      \msg_warning:nn { cfrac } { wrong-syntax }
      \xcontfrac[#1]{;#2}
     }
     {
      \tl_if_empty:nTF { #2 }
        { \cfrac_map_aux:nn { #1 } { \exp_not:N \use_none:n , #3 } }
        { \cfrac_map_aux:nn { #1 } { #2 , #3 } }
     }
  }
\cs_new:Npn \cfrac_map_aux:nn #1 #2
  {
   \clist_map_inline:nn { #2 }
     {
      \tl_put_right:Nn \l_cfrac_left_tl { \cfrac_begin:nn { #1 } { ##1 } }
      \tl_put_right:Nn \l_cfrac_right_tl { \exp_not:N \c_cfrac_rbrace_tl }
     }
   \tl_set:Nx \l_cfrac_left_tl
     { \l_cfrac_left_tl \c_cfrac_strut_tl \l_cfrac_right_tl }
   \tl_set:Nx \l_cfrac_left_tl { \l_cfrac_left_tl }
   \exp_after:wN \use_none:nnnnnn \l_cfrac_left_tl
  }
\cs_new:Npn \cfrac_begin:nn #1 #2
  {
   \exp_not:n
     { + \exp_not:N \cfrac[#1] { 1 } \c_cfrac_lbrace_tl \exp_not:N \mathstrut #2 }
  }
\ExplSyntaxOff

\usepackage{dsfont}
\usepackage{amsmath}
\usepackage{amssymb}
\usepackage{graphicx}

\usepackage{amssymb}

\usepackage[all,cmtip]{xy}

\usepackage{amsfonts}

\usepackage{soul}

\usepackage{mathpazo}
\usepackage{color}
\usepackage{yfonts}

\usepackage{paralist}

\usepackage{stmaryrd}
\usepackage{hyperref}

\usepackage{amsxtra}

\def\Int{{\rm int}}

\def\Id{{\rm Id}}

\def\id{{\rm Id}}

\def\cU{\mathcal U}
\def\cV{\mathcal V}

\def\cA{\mathcal A}

\newcommand{\cR}{\mathcal R}
\newcommand{\cT}{\mathcal T}
\newcommand{\cL}{\mathcal L}

\newcommand{\norm}[1]{\| #1\|}
\def\b{          \beta}

\def\a{          \alpha}

\def\B{          \mathcal B}

\def\cA{          \mathcal A}

\def\cB{          \mathcal B}
\def\cC{          \mathcal C}
\def\cD{          \mathcal D}

\def\cG{          \mathcal G}
\def\cH{          \mathcal H}

\def\clb{   \color{black}}

\def \R{{\mathbb R}}

\def \Z{{\mathbb Z}}
\def \N{{\mathbb N}}

\def \l{{\lambda}}
\def \a{{\alpha}}

\newcommand{\A}{{\mathbb A}}

\newcommand{\prf}{{\begin{proof}}}
\newcommand{\epf}{{\end{proof}}}

\newcommand{\G}{{\mathcal G}}

\newcommand{\Q}{{\mathbb Q}}
\newcommand{\D}{{\mathbb D}}

\newcommand{\ary}{\begin{eqnarray}}
\newcommand{\eary}{\end{eqnarray}}

\newcommand{\aryst}{\begin{eqnarray*}}
\newcommand{\earyst}{\end{eqnarray*}}

\newcommand{\enmt}{\begin{enumerate}}
\newcommand{\eenmt}{\end{enumerate}}

\DeclareMathOperator{\diff}{Diff}

\newtheorem{lemma}{\sc lemma}[section]

\newtheorem{prop}[lemma]{\sc Proposition}

\newtheorem{corollary}{\sc corollary}
\newtheorem{cor}[lemma]{\sc corollary}

\theoremstyle{definition}

\def\bee{\begin{equation}}
\def\eee{\end{equation}}

\newtheorem{defi}{\sc Definition}[section]

\theoremstyle{rema}

\newtheorem{rema}{\sc Remark}

\newcommand{\pdvr}[2]
{\dfrac{\partial^{#2} #1}{\partial \theta^{#2_1} \partial r^{#2_2}}}
\newcommand{\pdvrs}[2]
{\partial^{#2} #1 /\partial \theta^{#2_1} \partial r^{#2_2}}
\newtheorem{thm}{\sc Theorem}

\numberwithin{equation}{section}

\definecolor{blue}{rgb}{0,0,1}
\definecolor{red}{rgb}{1,0,0}
\definecolor{green}{rgb}{0,.7,0}

\author{Barney Bramham}
\address{Ruhr University Bochum}
\email{barney.bramham@rub.edu}

\author{Zhiyuan Zhang}
\address{CNRS, Institut Galil\'ee,
	Universit\'e Paris 13}
\email{zhiyuan.zhang@math.univ-paris13.fr}

\begin{document}

\title[On pseudo-rotations of the annulus with generic rotation number]{On pseudo-rotations of the annulus\\ with generic rotation number}
\date{\today}

\maketitle

\begin{abstract}	
	We show that for a Baire generic rotation number $\alpha \in \mathbb{R} / \mathbb{Z}$, the set of area preserving $C^\infty$-pseudo-rotations of the annulus with rotation number $\alpha$ equals the closure of the set of area preserving $C^\infty$-pseudo-rotations which are smoothly conjugate to the rotation $R_{\alpha}$. As a corollary, a $C^\infty$-generic area preserving pseudo-rotation of the annulus with a Baire generic rotation number $\alpha$ is weakly mixing.
\end{abstract}

\section{Introduction} 

\subsection{Results} 
Our results, which were announced in the survey article \cite{BZ}, concern smooth and area preserving pseudo-rotations on the closed annulus $\A = \R / \Z \times [0,1]$, which we will henceforth refer to simply as 
{\it pseudo-rotations} and denote by: 
\aryst
				F^{\infty}_{\A} := \Big\{ f \in \diff_0^\infty(\A, \omega) \mid \mbox{$f$ has no periodic points} \Big\}
\earyst
where $\omega$ is the standard smooth area form on $\A$ and $\diff_0^\infty(\A, \omega)$ denotes those area preserving diffeomorphisms that are isotopic to the identity.   
An obvious subset of the pseudo-rotations are the {\it true rotations} with irrational rotation number. 
By the true rotations we mean:  
\aryst
				O^{\infty}_{\A} := \Big\{ hR_t h^{-1} \mid h  \in \diff^\infty(\A, \omega),\ t\in\R/\Z \Big\}
\earyst
where $R_t$ is the rigid rotation $([x],y)\mapsto([x+t],y)$.  

%
%
%
%
Anosov and Katok \cite{AK70} discovered the conjugation method, which produces pseudo-rotations with interesting dynamics, i.e.\ that are not true rotations.        
By construction, their method yields pseudo-rotations as limits of true rotations, i.e.\ as elements of $$\overline{O^{\infty}_{\A} }\cap F^{\infty}_{\A},$$   
so that one could refer to this set as the {\it Anosov-Katok pseudo-rotations}, \cite{JS}.  
All known pseudo-rotations are Anosov-Katok pseudo-rotations, see e.g.\cite{AK70, FK, FS, LeRS},  and a natural question is whether this is true in general.   Our main result 
confirms this in a large number of cases.   

To make our statement precise, recall that each pseudo-rotation 
$f:\A\to\A$ admits a {\it rotation number} $$\rho(f) \in (0,1) \setminus \Q$$ 
which can be seen as the mean asymptotic rotation number of every trajectory about the annulus, see Definition \ref{def rotation number}, and denote the set of all {\it pseudo-rotations with rotation number $\a \in (0,1) \setminus \Q$} by
\aryst
				F^{\infty}_{\A}(\a) := \Big\{ f \in F^\infty_\A \mid \rho(f) = \a \Big\}.   
\earyst
This contains the {\it true rotations with rotation number $\a$}: 
\aryst
				O^{\infty}_{\A}(\a) := \Big\{ hR_\a h^{-1} \mid h  \in \diff^\infty(\A, \omega) \Big\}. 
\earyst 
Our main result is that for a Baire generic rotation number, all pseudo-rotations are Anosov-Katok pseudo-rotations.  Infact our result is slightly stronger, since we can obtain this using the closure of the space of true rotations with {\it fixed} rotation number.  More precisely: 
\begin{thm}\label{main thm F = closure O}
	There exists a Baire subset $\cA\subset (0,1) \setminus \Q$ of rotation numbers with the following property: For all $\a \in \cA$ 
	\ary\label{E: F = closure O}
	F^{\infty}_{\A}(\a) = \overline{O^{\infty}_{\A}(\a)}
	\eary
	where the closure is taken in $\diff^\infty(\A)$ with respect to the $C^\infty$-topology.
\end{thm}
Infact $\cA$ will be a subset of the Liouville numbers, which we recall is itself a Baire subset of $(0,1) \setminus \Q$ with measure zero. 
Note that we do not need to write $\overline{O^{\infty}_{\A}(\a)}\cap F^{\infty}_{\A}(\a)$ on the right hand side of \eqref{E: F = closure O} because 
of the inclusion $\overline{O^{\infty}_{\A}(\a)}\subset F^{\infty}_{\A}(\a)$, which, while not entirely trivial, 
is not difficult to prove using Franks' theory of positively and negatively returning discs \cite{F88}.  
Indeed this inclusion holds for all irrational $\a$, not just some Baire subset, and even the $C^0$-version holds; for a proof of these claims see \cite[Proposition 33]{B2}.    

Thus, this article is devoted to proving the opposite inclusion  
\ary\label{F subset closure O}
				F^{\infty}_{\A}(\a) \subset \overline{O^{\infty}_{\A}(\a)}   
\eary
namely, that each pseudo-rotation with rotation number in $\cA$ can be $C^\infty$-approximated by true rotations with the same rotation number.  

It is interesting to note that for Diophantine $\a$ there is possibly more reason to expect \eqref{F subset closure O} to hold, since it would support a positive answer to a long standing question of Herman, see below.    However, \eqref{F subset closure O} is currently not known for any Diophantine rotation numbers.

%
%
%
%

Combining Theorem \ref{main thm F = closure O} with a result from \cite{AK70} we 
have the following interesting corollary:

\begin{corollary}\label{cor wm}
For a Baire generic $\a \in (0,1) \setminus \Q$, 	the set of weakly mixing pseudo-rotations in $F^{\infty}_{\A}(\a)$ forms a Baire set with empty interior, with respect to the $C^\infty$-topology.   
\end{corollary}
\begin{proof}
	On the one hand, Anosov and Katok, see \cite{AK70}, show that for a Baire generic $\a$, weak mixing is a $C^\infty$-generic property in $\overline{O^{\infty}_{\A}(\a)}$.   Thus by Theorem \ref{main thm F = closure O} 
	weak mixing is a $C^\infty$-generic property in $F^{\infty}_{\A}(\a)$.  
	
	On the other hand, the second statement follows since elements of $O^{\infty}_{\A}(\a)$ are never weak mixing and by Theorem \ref{main thm F = closure O} the complement $F^{\infty}_{\A}(\a)\backslash O^{\infty}_{\A}(\a)$ 
	has empty interior. 
\end{proof}
\begin{rema}
	Corollary \ref{cor wm} is rather sharp in the following sense:
	\enmt
	\item The genericity of $\a$  cannot be improved to positive Lebesgue measure. This follows from the KAM result of Fayad-Krikorian \cite{FKrik} (attributed by the authors to Herman), that a neighborhood of $R_\a$ in 	$F^\infty_{\A}(\a)$ lies in $O^\infty_{\A}(\a)$, for any Diophantine $\a$.
	\item Weakly mixing cannot be replaced by mixing. In fact, it follows from the proof of \cite{Br} and \cite{A} that for a Baire generic $\a$, $F^\infty_\A(\a)$ contains no topologically mixing maps. 
	\eenmt
\end{rema}

\subsection{Related literature}
We recall here how Theorem \ref{main thm F = closure O} relates to known results and open questions on pseudo-rotations.   We also explain an analogy to work of 
Herman and Yoccoz on circle diffeomorphisms which partially inspired the proof in this article.   

\subsubsection{A question of Katok and the $C^0$-approximations in \cite{B2}}
Problem 1 in \cite{Kopen} asks whether every area preserving surface diffeomorphism with zero topological entropy is a limit of integrable systems.
The result in \cite{B2} is a $C^0$-approximation result for pseudo-rotations of the closed $2$-dimensional disc.   The approach used pseudo-holomorphic curve methods, i.e.\ 
completely different to those in the current article.   

The precise statement in \cite{B2} was as follows: Any pseudo-rotation $f$ of the disc, can be $C^0$-approximated by a smooth diffeomorphism $f_k$, where each $f_k$ 
can be smoothly conjugated to a rigid rotation $R_{p_k/q_k}$, where $p_k/q_k\in\Q$ is close to the rotation number of $f$.   
See also \cite{L} for a similar result.  


In constrast our result provides a much stronger approximation scheme in the following three senses: 
\begin{itemize}
	\item The convergence $f_k\to f$ is in the $C^\infty$-topology rather than just $C^0$.   
	\item Each $f_k$ is area preserving.  
	\item Each $f_k$ has the same rotation number as $f$ itself.  
\end{itemize}
The one sense in which the result in \cite{B2} is stronger, is that it applies without restriction on the (irrational) rotation number of the pseudo-rotation.   

\subsubsection{The weak mixing pseudo-rotations of \cite{AK70} and \cite{FS}}
Another motivation for this work is the important extensions in Fayad-Saprykina \cite{FS}, see also Fayad-Katok \cite{FK}, of the above mentioned work of 
Anosov and Katok \cite{AK70}.   Indeed, they construct examples of Anosov-Katok pseudo-rotations that are weakly mixing, for each Liouville rotation number.     
Curiously,  it is not clear whether the set of weakly mixing pseudo-rotations is generic.
   However, it is automatic from Theorem \ref{main thm F = closure O}, that this must be the case when 
$\a$ lies in the possibly smaller Baire-generic set of Liouville rotation numbers $\cA$ produced by Theorem \ref{main thm F = closure O}.  

\subsubsection{Herman's question on Diophantine pseudo-rotations}
In \cite{H98} M.\ Herman asks whether every smooth pseudo-rotation of the annulus with Diophantine rotation number, i.e.\ irrational and non-Liouvillean, is conjugate 
to a rigid rotation.   For smooth conjugacies this is the question whether 
\aryst
			F_{\A}^\infty(\a) = O_{\A}^\infty(\a)
\earyst
for all Diophantine $\a$.   
There are two partial results, using KAM methods, that we have referred to above: Herman showed that if $f$ is a Diophantine pseudo-rotation 
then there are invariant circles near the boundary filling up a subset of positive measure.    This implies in particular that the results in \cite{FS} are sharp.   
The second result is referred to by the authors as Herman's last geometric theorem: Namely, Fayad and Krikorian \cite{FKrik} showed that if $f$ is a Diophantine 
pseudo-rotation that is sufficiently close to the rigid rotation $R_{\rho(f)}$, then $f$ is conjugate to $R_{\rho(f)}$ via a smooth area preserving diffeomorphism.     

\subsubsection{The Herman-Yoccoz Theorems on circle diffeomorphisms}
Theorem \ref{main thm F = closure O} can be seen as a natural analogue of the Liouville part of the following well-known theorem of Herman and Yoccoz in \cite{H,Yo95}: 

\begin{thm}[Herman-Yoccoz]\label{thm HY}
	For any irrational $\a \in \R/\Z$, we denote by $F^\infty(\a)$ the set of $C^\infty$ circle diffeomorphisms with rotation number $\a$, and denote by $O^\infty(\a)$ the set of $C^\infty$ circle diffeomorphisms which are $C^\infty$-conjugate to the standard rotation $R_\a$.  
Then we have
	\ary\label{E: HY}
	F^\infty(\a) = \overline{O^\infty(\a)} 
	\eary
	for all $\a$ irrational,  
	where the closure is in the $C^\infty$-topology.
\end{thm}

The proof of \eqref{E: HY} splits into two cases, depending on whether $\a$ is Diophantine or Liouvillean (denoted $\cD$ and $\cL$ respectively).   The Diophantine case is a trivial consequence 
of the Diophantine rigidity phenomenon: 
	\ary\label{E: HY 2}
	F^\infty(\a) = O^\infty(\a)\qquad \forall \a\in\cD
	\eary
which is itself a famous result of Herman \cite{H} for a full-measure subset of $\cD$, 
and which was later extended to all of $\cD$ by Yoccoz  \cite{Yo95}.

The Liouville part of Theorem \ref{thm HY} was conjectured by Herman in \cite[Conjecture 7.1]{H}. In fact Herman already showed that Theorem \ref{thm HY} holds for a Baire generic set of $\a$, see \cite[Theorem 7.3]{H}. However, his proof was based on the Diophantine part of Theorem \ref{thm HY} (at least for a full measure set of $\a$), and used certain properties of the function $t \mapsto \rho(R_t f)$ of a circle diffeomorphism $f$. 

The full answer to Herman's conjecture - about the Liouville part of Theorem \ref{thm HY} - was provided by Yoccoz in \cite{Yo95}. Yoccoz showed that any $C^\infty$-smooth circle diffeomorphism 
with a Liouville rotation number can be $C^\infty$-approximated by a {\it quasi-rotation}: this is a class of circle diffeomorphisms which, among other things, admits a renormalization that {\it is} a standard rotation. 

Our proof of Theorem \ref{main thm F = closure O} is somewhat similar to this proof of Yoccoz: We also consider certain renormalizations of a pseudo-rotation. However, the type of estimates are very different. We are unable to transfer the strong estimates for circle diffeomorphisms, such as Denjoy's inequality in \cite{Yo84}, to general pseudo-rotations, due to the possible occurrence of complicated geometry which does not appear in dimension $1$. On the other hand, the area-preserving hypothesis provides us with certain strong $C^0$-estimates established in \cite{A}. Combining such estimates with a suitable arithmetic condition, we are able to extract some useful information from a sequence of suitably renormalized pseudo-rotations.

%
%
%
%

It is still an open question of Herman \cite{H98} whether the Diophantine rigidity seen in Theorem \ref{thm HY} holds for pseudo-rotations.   Moreover, it is unclear how to deform a general pseudo-rotation within say the set of pseudo-rotations and true rotations, and change the rotation number.  
This blocks a direct generalisation of Herman's approach for pseudo-rotations.   

We should also point out that the study of pseudo-rotations can be essentially traced back to the question of Birkhoff \cite{Bir} (see also \cite{H98}) as to whether there are non-trivial analytic diffeomorphisms of the $2$-sphere with $2$ fixed points, i.e.\ whether $F^{\omega}_{S^2}= O^{\omega}_{S^2}$.   That these spaces can be different was recently announced by Berger \cite{Ber}.

Recently, Avila and Krikorian have announced \footnote{See the minicourse of Krikorian in the program \lq\lq Renormalization and universality in Conformal Geometry, Dynamics, Random Processes, and Field Theory \rq\rq in 2020 at Simons Center for Geometry and Physics.} an improvement of Theorem \ref{main thm F = closure O}: For every non-Brjuno $\alpha$, one has  $F^{\infty}_{\A}(\a) = \overline{O^{\infty}_{\A}(\a)}$.    (The non-Brjuno numbers form a strict subset of the Liouville numbers that is strictly larger than the set $\cA$ in Theorem \ref{main thm F = closure O}).   Moreover, they have announced the following result:  For every pseudo-rotation $f$ in an open neighborhood of the rigid rotations on the closed disc, there exists a sequence of area-preserving diffeomorphism $h_n$ such that $h_n f h_n^{-1}$ converges to $R_{\rho(f)}$ in the $C^\infty$ topology.
Their method involves delicate estimates on high iterates of the maps, while our method for getting this weaker result relies only on rather soft estimates.

\subsection{Outline of the construction}\label{S:overview} 
We outline here our construction of the integrable approximations that will prove \eqref{F subset closure O} and hence Theorem \ref{main thm F = closure O}.  

Let $r\in\N, M>0$ and $\epsilon>0$ be arbitrary and fixed.     Consider a pseudo-rotation $f:\A\to\A$ with rotation number $\alpha\in(0,1)\backslash\Q$ and 
$$\|f\|_{\diff^{r+6}(\A)}\leq M.$$    
If $\a$ satisfies a Liouville-type condition that depends only on $(r,M,\epsilon)$, we will modify $f$ in a $\diff^{r-2}(\A)$-small sense as measured by $\epsilon$, to a map that is conjugate to the rigid rotation $R_{\a}$.   The condition on $\a$ will be of the form described in Lemma \ref{lem construct rotation nbs}, for some function $P=P_{(r,M,\epsilon)}$, hence actually open and dense in $(0,1)\backslash\Q$.    Taking a countable intersection over $(r,M,\epsilon)$-tuples leads to a Baire generic condition on $\a$.  

The denominators of the best rational approximations  $p_1/q_1,\ p_2/q_2,\ldots$ to $\alpha$ play a crucial role, see \ref{S:rational} for definitions.   
If $\alpha$ is sufficiently Liouvillean then $f^{q_n}$ will be close to the identity, see Theorem \ref{thm aflxz} below (see also \cite{A}, or \cite{B2} for the $C^0$-version).   
In particular if $\gamma$ is any embedded curve in the annulus that connects the two boundary circles, then 
$f^{q_{n+1}}(\gamma)$ will be close to $\gamma$.    Just how close, depends on the smallness of the fractional part $\{q_{n+1}\alpha\}$ of $q_{n+1}\alpha$ (depending also on $r, M$ and $\epsilon$), which in turn can be formulated in terms of a sufficiently large gap between $q_{n+2}$ and $q_{n+1}$.   
%
%
The first significant step in the construction is to find a curve $\gamma=\gamma_n\subset\A$ as above, so that the iterates 
\ary\label{curves}
\gamma,\ f(\gamma),\ldots, f^{q_{{n+1}}-1}(\gamma)
\eary 
are mutually disjoint, assuming large enough gaps in the sequence $(q_j)$ near to $q_n$.  This is achieved in Sections \ref{S:Brouwer} and \ref{S:renorm} using a renormalization procedure, see 
Theorem \ref{cor smooth return domain new version}.  
%
%

The (closure of each component of) the complement of the curves in \eqref{curves} defines a tiling of $\A$ by topological rectangles 
\ary\label{tiling 0}
	\cT_0,\ldots,\cT_{q_{{n+1}}-1}
\eary 
having mutually disjoint interiors.   
The curves in \eqref{curves} would be cyclically permuted by $f$ if $f^{q_{n+1}}(\gamma)$ coincided exactly with $\gamma$.   
This cannot happen, but $f^{q_{n+1}}(\gamma)$ is very close to $\gamma$ since $f^{q_{n+1}}$ is close to the identity.    
Our strategy therefore is to modify $f$, near to $f^{q_{n+1}-1}(\gamma)$, to obtain a new map $g$, so that the $g$-orbit of $\gamma$ closes up after $q_{n+1}$-iterates, 
see Figure \ref{F:f to g}.   
Then the tiling \eqref{tiling 0} is $g$-invariant, where $g$ is close to $f$.    Each element of the tiling will be a fundamental domain for $g$ and there is a well defined first return map $g^{q_{n+1}}:\cT_0\to\cT_0$, see \ref{S:decomposition}.  

Moreover, we can arrange that $g^{q_{n+1}}(x)$ is close to $f^{q_{n+1}}(x)$, for each $x\in\cT_0$, which in turn is close to $x$ because $f^{q_{n+1}}$ is close to the identity.   
Fine tuning the deformation from $f$ to $g$ we can arrange that    
\aryst
					g^{q_{{n+1}}}=\id, 
\earyst 
i.e.\ that $g$ is a periodic approximation to $f$, see sections \ref{S:g periodic} and \ref{S:periodic approx}.

The next goal is to replace $g$ by an approximation to $f$ that has the same (irrational) rotation number as $f$.   The idea is that since $g$ is $q_{n+1}$-periodic we can write 
$g=h^{-1}R_{p_{n+1}/q_{n+1}}h$ for some change of coordinates $h\in\diff^\infty(\A)$.    In \ref{S:h} we show this can be achieved with crucial bounds 
on $h$ in a $\diff^r(\A)$-sense, bounds which ultimately come from the renormalization argument in Section \ref{S:renorm}.   
From these bounds on $h$ it will follow that $g=h^{-1}R_{p_{n+1}/q_{n+1}}h$ is close to $h^{-1}R_{\alpha}h$ in $\diff^{r-1}(\A)$.  Here again our Liouville-type condition is important, because $|p_{n+1}/q_{n+1} -\alpha|$ decreases as $q_{n+2}$ increases, while $\|h\|_{\diff^r(\A)}$ remains bounded.   The upshot, see Proposition \ref{P: approx with rot a}, is that $$h^{-1}R_{\alpha}h$$ is an $\epsilon$-approximation to $g$, and hence also to $f$, in $\diff^{r-1}(\A)$, 
provided the rotation number $\a$ is such that $q_{n+2}$ is sufficiently larger than $q_{n+1}$, depending also on $r,M$ and $\epsilon$.    
In a final step, Proposition \ref{P: area preserving}, we modify $h$ to make $h^{-1}R_{\alpha}h$ area preserving, and therefore an element of $O^{\infty}_{\A}$ that $\epsilon$-approximates $f$ 
in $\diff^{r-2}(\A)$, as required.    
%
\subsection*{Acknowledgements}

Z.Z. would like to thank Artur Avila and Rapha\"el Krikorian for discussion on one occasion.
Z.Z. would also like to acknowledge the online talk by Rapha\"el Krikorian during the Workshop \lq\lq Between Dynamics and Spectral Theory \rq\rq at  the 
Simons Center for Geometry and Physics back in 2016, which inspired this article. 
This work was initiated in 2019 while the authors were at the Institute for Advanced Study both supported by the National Science Foundation under Grant No.\ DMS-1638352.  
We thank them for their hospitality and excellent working environment.  
B.B. was also partially supported by the SFB/TRR 191 `Symplectic Structures in Geometry, Algebra and Dynamics', funded by the DFG (B1 281071066 -- TRR 191).


\section{Preliminaries}\label{sec prelim}
\subsection{Best rational approximations}\label{S:rational}
Let us  recall a few well known facts about the best rational approximations to $\a\in (0,1)\setminus \Q$ and its continued fraction expansion.   Readers can consult \cite[Chapters X, XI]{HW} for more details.   

For $x\in\R$ we will write 
\aryst 
		\lfloor x\rfloor &:= &\max\{ n\in\Z\ |\ n\leq x\} \in \Z, \\
		q(x) &:= &\lfloor 1/x\rfloor \in \Z, \quad  x \neq 0.
\earyst
for the integer parts of $x$ and $1/x$ respectively, and denote the fractional part of $x$ by
\ary\label{E:frac}
				\{x\}:=x-\lfloor x\rfloor\in[0,1).
\eary
The Gauss map $\G: [0,1)\rightarrow[0,1)$ is defined by
\aryst
			\G(x)=\frac{1}{x} - q(x)   
\earyst
on $(0,1)$ and $\G(0):=0$.   If $x$ is irrational then so is $\G(x)$.
 
\begin{defi}
	For any $\a \in (0,1) \setminus \Q$, we define the sequences $(\a_n)_{n \geq 0}$ and $(\beta_n)_{n\geq0}$ in $(0,1) \setminus \Q$ by 
	\aryst
	\a_0 := \a, \qquad \a_n &:=& \cG^n(\a_0)\qquad \forall n \geq 1,\\
	\beta_n &:=&\prod_{i=0}^{n}\a_i \qquad \forall n \geq 0.  
	\earyst 
	Furthermore we define sequences of non-negative integers $(a_n)_{n \geq 0}$, $(q_n)_{n \geq 0}$  as follows: 
	\ary
	a_0 := 0, \qquad  \quad a_n &:=& q(\a_{n-1})\qquad \forall n \geq 1,  \label{eq a n 0}\\ 
	q_{0}:=1,\quad q_1:=q(\alpha), \quad  q_{n+2} &:=& q_n + q_{n+1} a_{n+2} \qquad \forall n \geq 1.  \label{q iteration}
	\eary
	We also define $(p_n)_{n \geq 0}$ by $p_0 = 0$, and for $n \geq 1$, define
	 $p_n$ to be the closest integer to $q_n\alpha$, which is unique by irrationality of $\alpha$.   Thus 
	 \ary\label{E:pn distance}
			\norm{q_n\a}_{\R / \Z}=d(q_n\a,\Z)=|q_n\a - p_n| \qquad \forall n\in\N.  
	\eary
	 It is easy to check that $q_1\a\in(1/2,1]$ always, so that $p_1=1$ always.     
	 
	 We will use the notation $\a_n(\a)$, $q_n(\a)$ and $p_n(\a)$ when it is necessary to indicate the dependence of the sequences on $\a$.  
\end{defi}
	Note that $\a_{n-1}^{-1} = a_n + \a_n$ for all $n\in\N$, 
	since $\a_n=\cG(\a_{n-1})=1/\a_{n-1}-q(\a_{n-1})=1/\a_{n-1}-a_n$,    
	which implies the following continued fraction expansions: 
	\aryst
			\a=
				\xcontfrac{;a_1,a_2,\cfracdots,a_n + \a_n}
			\qquad\qquad
			\frac{p_n}{q_n}=
				\xcontfrac{;a_1,a_2,\cfracdots,a_n}
	\earyst
	The following well known identities will be useful: 
	\ary
		p_{n+1}q_n-p_nq_{n+1}&=& (-1)^n \label{relation pn qn}\\
	(-1)^n(q_n \a - p_n)  &=&  \beta_n > 0 \label{sign}
	\eary
	for all $n\in\N$.   In particular for $n \geq 1$ we have $\beta_n=|q_n\a - p_n| =\|q_n\a\|_{\R/\Z}$.   
	For any $n \geq 1$,
	 the integers $p_n$ and $q_n$ are relatively prime, and     
	$p_n / q_n$ is called the $n$-th best rational approximation of $\a$.   The following estimates are also well known:
	\ary \label{eq continuefraction} 
	 \frac{1}{2q_{n+1}} < \frac{1}{q_n + q_{n+1}} < \b_n < \frac{1}{q_{n+1}}, \\
	\label{eq anupperlower}
	 \a_n,\  q(\a_n)^{-1} \in \left(\frac{q_n}{2q_{n+1}}, \frac{2q_n}{q_{n+1}}\right).
	\eary
	These lead to two further standard inequalities: 
\ary
						\G(\a_n)&<&\frac{2q_{n+1}}{q_{n+2}},\label{E:Gauss of alpha}
\eary
from $\G(\a_n)=\a_{n+1}$ and \eqref{eq anupperlower} 
\ary
						\G(\beta_{n-1}) &<& \frac{2q_{n-1}q_{n}}{q_{n+1}}.   \label{E:Gauss of beta}
\eary
To see \eqref{E:Gauss of beta}: \eqref{relation pn qn} implies 
$1/\beta_{n-1} - q_n = q_{n-1}\a_n < 2q_{n-1}q_{n}/q_{n+1}$.   
So if the latter ratio is less than $1$, then$\left\lfloor 1/\beta_{n-1}\right\rfloor = q_n$ 
giving $\G(\beta_{n-1})=1/\beta_{n-1}-q_n= q_{n-1}\a_n$, which implies \eqref{E:Gauss of beta}.
	Finally, we will also require: 
\begin{lemma}\label{lem construct rotation nbs}
		For any  function $P: \N \to \R_+$, the set 
		\aryst
		\ \ \cC_P := \big\{ \alpha \in (0,1)\setminus \Q &\mid& \exists n\geq 3 \mbox{ odd such that } q_{n} > P(q_{n-1}), \\
		&&q_{n+1} > P(q_{n}),\  \ q_{n+2} > P(q_{n+1})  \big\}
		\earyst 
		is open and dense in $(0,1)\setminus \Q$.
\end{lemma}
\begin{proof} This follows from standard observations in the theory of rational approximations, so we merely outline the idea.    
Indeed, $\cC_P$ is the union over $n\in\N_{\mathrm{odd}}$ of the sets 
\aryst
 \cC_P(n) &=  & \Big\{ \alpha \in (0,1)\setminus \Q \Big| \ q_{n} > P(q_{n-1}),\ \ q_{n+1} > P(q_{n}),\ \ q_{n+2} > P(q_{n+1})  \Big\}.
\earyst 
Each $\cC_P(n)$ is open in $(0,1)\setminus \Q$ because each function $x\mapsto q_n(x)$ is continuous at irrational $x$.    
One can also show that each $\cC_P(n)$ is $\mu^n$-dense in $[0,1]$ for some $\mu\in(0,1)$.    
Hence $\cC_P$ is open and dense in $(0,1)\setminus \Q$.   
\end{proof}
 
\subsection{The annulus and lifts}\label{SS:lifts}
The universal covering of $\A:=\R/\Z\times[0,1]$ is $\tilde \A = \R \times [0,1]$.   We write  
 $$\pi : \tilde \A \to \A$$ 
for the natural projection, and $T: \tilde \A \to \tilde \A$ for the deck transformation $T(x,y) = (x + 1, y)$.
Recall that every homeomorphism $f:\A\to\A$ has a lift 
$F:\tilde\A\rightarrow\tilde\A$, also a homeomorphism, unique up to composition by a power of $T$, satisfying $\pi F = f \pi$.   
Any lift commutes with $T$.  

Let $\omega=dx\wedge dy$ be the standard smooth area form on $\A$. 
By a slight abuse of notation, we also denote by $\omega$ the lift to an area form on $\tilde \A$.
If $f$ is an $\omega$-preserving diffeomorphism then so is each lift $F$ an $\omega$-preserving diffeomorphism.    

The boundary components of $\A$ and $\tilde\A$ will be denoted $B_i:=\R / \Z \times \{i\}$, respectively $\tilde B_i :=\R  \times \{i\}$, for $i=0,1$. 
For  $\Omega\subset\A$, $\Int(\Omega)$ denotes the interior \textit{with respect to the subset topology} of $\A$.   Similarly for subsets of $\tilde\A$.

\subsection{The $C^0$-topology.}\label{SS:C0}
We equip $\tilde\A= \R \times [0,1]$ with the Euclidean metric as a subset of $\R^2$    
and give the annulus $\A$ the unique metric $d_{\A}$ that makes the projection $\pi:\tilde\A\to\A$ a local isometry.    Explicitely, if we set 
\ary\label{E:circle norm}
			\norm{x}_{\R / \Z} := d(x, \Z)\in [0,1/2]\qquad\forall x\in\R,  
\eary
then $d_{\A}$ is the product of the induced metric on $\R / \Z$ with the Euclidean one on $[0,1]$.    (I.e.\ $d_{\A}$ 
is the distance metric induced by the Riemannian metric on $\A$ from the Euclidean scalar product on $\tilde\A\subset\R^2$ via $\pi$.) 
Then set
\aryst
						d_{C^0(\Omega)}(f,g)\ :=\ \sup_{x\in\Omega}d_{\A}\big(f(x),g(x)\big) 
\earyst
is the usual $C^0$-distance for maps $f,g: \Omega \to \A$, from arbitrary $\Omega$, and 
\aryst	
		\norm{F} :=\norm{F}_{\Omega} := \sup_{x \in \Omega}|F(x)|
\earyst 
where $F=(F_1,F_2): \Omega \to \tilde\A$ and $|\cdot |$ is the Euclidean norm on $\R^2$.   
If $\Omega=\A$ and $f,g:\A\to\A$ are homeomorphisms, and $F, G$ are any lifts of $f, g$ respectively, then one finds $d_{\A}\big(f(x),g(x)\big)\ =\ \min_{j\in\Z}|T^jF(\tilde x)-G(\tilde x)|$ 
for each $x\in\A$, where $\pi(\tilde x)=x$.   Hence 
\ary\label{E:A metric 2}
						d_{C^0(\A)}(f,g)\ \leq\ \|F-G\|_{[0,1]^2}.  
\eary
\begin{lemma}\label{lem:lift}
	If $f:\A\to\A$ is a homeomorphism with $d_{C^0}(f,\Id_{\A})<1/2$ then there is a unique lift $F:\tilde\A\to\tilde\A$ satisfying 
	$\norm{F - {\rm Id}_{\tilde\A}}_{\tilde\A}<1/2$.   This lift satisfies 
	\ary\label{lift}
	|F(\tilde x) - \tilde x | \ =\ d_{\A}\big(f(x),\,x\big)  
	\eary
	for all $x\in\A$ and all $\tilde x\in\tilde\A$ with $\pi(\tilde x)=x$.   Hence $\norm{F - {\rm Id}_{\tilde\A}}_{\tilde\A}=d_{C^0}(f,\Id_{\A})$.  
	This lift $F$ is also characterized by $\rho(F)\in(-1/2,1/2)$, see \eqref{def rotation number F}.  
\end{lemma} 
The proof just uses that $\pi$ is an isometry on balls of diameter $<1/2$.    
%
%
%

\subsection{$C^r$ measurements.}\label{SS:Cr}
Let $r\in\N$ and $U\subset\R^2$ be open, or $U=\tilde\A$ or $U=\A$.  For any $C^r$-smooth map $F: U \to \tilde\A$ define 
\aryst\label{D:Cr }
 \norm{D^rF}_{U} := \max_{|\a|= r}\norm{\partial^\a F}_{U},\quad \norm{F}_{C^r(U)} := \max_{1\leq |\a|\leq r}\norm{\partial^\a F}_{U},
\earyst
\aryst
	\norm{F}_{\diff^r(U)} := \max\Big\{ \norm{F}_{C^r(U)},\ \norm{F^{-1}}_{C^r(V)}\Big\},
\earyst
where the latter is for a $C^r$-diffeomorphism $F:U\to V$ between 
open subsets of $\A$, $\tilde\A$ or of $\R^2$.    Despite notation, the above are clearly not norms.

If the target is $\A$ then we define $C^r$-measurements locally, since $\A$ is not isometric to a subspace of the linear space $\R^2$.    
Indeed, if $f: U \to \A$ is $C^r$-smooth from $U\subset \A$ open, then for each $z\in U$ and each multi-index $\a$ with $1\leq |\a|\leq r$, we identify 
$\partial^\alpha f(x)\ :=\ \partial^\alpha F(\tilde x)\in\R^2$ 
where $F:\mathcal{O}\subset\tilde\A\to\R^2$ is any ``local lift'' of $f$ to an open neighborhood of a point $\tilde x\in\pi^{-1}(x)$, meaning $F:=\hat{\pi}^{-1}\circ f\circ\hat{\pi}$ where $\hat{\pi}$ is the restriction of $\pi$ to $\mathcal{O}\subset\tilde\A$, and the latter is sufficiently small that 
$\hat{\pi}$ is an isometry onto its image.   This is independent of the choices since $T$ is an isometry.   

Now for each $C^r$-smooth map $f: U \to \A$ from an open subset $U\subset\A$, we define 
the following quantities which, again despite the notation, are clearly not norms; 
\aryst
				\|D^rf\|_{U} := \max_{|\a|=r}\sup_{x\in U}|\partial^\a f(x)|,   \qquad \|f\|_{C^r(U)} := \max_{1 \leq j \leq r}\|D^jf\|_{U}. 
\earyst 
In the special case $U=\A$, so $f$ has a global lift $F:\tilde\A\to\tilde\A$, one shows 
\aryst
			 \|D^rf\|_{\A}&=&\|D^rF\|_{[0,1]^2}=\|D^rF\|_{\tilde\A}\\  
			\|f\|_{C^r(\A)}&=&\|F\|_{C^r([0,1]^2)}=\|F\|_{C^r(\tilde\A)}.  
\earyst
If $f:U\to V$ is a diffeomorphism between open subsets of $\A$ 
then we set 
\aryst
				\|f\|_{\diff^r(U)} :=  \max\Big\{ \|f\|_{C^r(U)},\ \|f^{-1}\|_{C^r(V)}\Big\}.
\earyst 
Finally   we define a $C^r$-measurement that includes the zeroth order term, and distinguish this by using the notation for a metric rather than a norm:
%

\begin{defi}

For each $r\in\N$ we define the following metric on $\diff^r(\A)$: 
\aryst\label{D:diffr metric}
		 &d_{\diff^r(\A)}(f,g) &:=  \max\Big\{d_{C^0(\A)}(f,g),\ d_{C^0(\A)}(f^{-1},g^{-1}),\nonumber \\ 
		 &				&\hspace{70pt} \norm{F-G}_{C^r(\tilde\A)},\ \norm{F^{-1}-G^{-1}}_{C^r(\tilde\A)}\Big\} 
\earyst
for all $f,g\in\diff^r(\A)$, where $F,G:\tilde\A\to\tilde\A$ are any lifts of $f, g$ respectively.  
\end{defi}
When there is no confusion, we further abbreviate $\norm{f}_{C^r(U)}$, $\norm{f}_{\diff^r(U)}$, $d_{\diff^r(U)}$, by $\norm{f}_{C^r}$, $\norm{f}_{\diff^r}$ 
and $d_{\diff^r}$ respectively.   

Occasionaly we use the H\"older spaces $C^{r,\theta}$ on $\A$ and $\tilde\A$ and on open subsets thereof, 
where $r\in\N$ and $\theta\in(0,1)$.    These can be defined analogously.  
 
A sequence of $C^\infty$-smooth diffeomorphisms on $\A$ is said to be convergent in $\diff^\infty(\A)$ if it converges with respect to $d_{\diff^r(\A)}$ for each $r\in\N$. 
 
\subsection{Estimates for compositions}\label{SS:compositions}
We will often implicitely use that $\|f\circ g\|_{C^r}$ is bounded by a function of $\|f\|_{C^r}$ and $\|g\|_{C^r}$, and hence 
$\|f\circ g\|_{\diff^r}$ by a function of $\|f\|_{\diff^r}$ and $\|g\|_{\diff^r}$.    
To get smallness of the composition $f\circ g$ in terms of smallness of $f$ or $g$, one needs slightly more:    

\begin{lemma}\label{L:composition}
Let $r_1,r_2\in\N_0$ and $\theta_1,\theta_2\in[0,1)$ and $M>0$, be given, so that $r_1+\theta_1<r_2+\theta_2$.    
Then for each $\epsilon>0$ there exists $\delta>0$, depending on $(r_1+\theta_1,r_2+\theta_2,M,\epsilon)$, 
so that if $d_{\diff^{r_2,\theta_2}(\A)}(\varphi,\id)\leq M$ and $d_{\diff^{r_1,\theta_1}(\A)}(\psi,\id)\leq \delta$, then 
\aryst
						d_{\diff^{r_1,\theta_1}(\A)}(\varphi\circ\psi,\varphi)<\epsilon.   
\earyst
Similarly $d_{\diff^{r_1,\theta_1}(\A)}(\psi\circ\varphi,\varphi)<\epsilon$. 
\end{lemma}
The proof is a standard argument using mainly that after the compact inclusion $C^{r_2,\theta_2}\hookrightarrow C^{r_1,\theta_1}$ the 
composition map becomes uniformly continuous on bounded subsets.  
%

\subsection{Brouwer curves and admissible coordinates}
We say that $\gamma\subset\A$ is a {\it simple curve} connecting $B_0$ and $B_1$ if $\gamma = \phi([0,1])$ for some continuous injective map $\phi : [0,1] \to \A$ that maps $0$, resp. $1$, into $B_0$, resp. $B_1$, and that maps $(0,1)$ into $\A\backslash\partial\A$.  	We define simple curves in $\tilde\A$ connecting $\tilde B_0$ and $\tilde B_1$ in a similar way.

\begin{defi}\label{ordering on curves}
	For a pair of simple curves $\gamma_1,\gamma_2$ in $\tilde\A$ we will say that \textit{$\gamma_1$ is to the left of $\gamma_2$}, or \textit{$\gamma_2$ is to the right of $\gamma_1$}, and write 
	\aryst
					\gamma_1 < \gamma_2 
	\earyst
	if $\gamma_1\cap\gamma_2=\emptyset$ and $\gamma_1$ lies in the component of $\tilde\A\backslash\gamma_2$ containing points with arbitrarily 
	negative $\R$-coordinate.   This defines a partial ordering on simple curves in $\tilde\A$.  
\end{defi}

\begin{defi}[Brouwer curves]\label{def:Brouwer}
	Let $f$ be a homeomorphism of $\A$.
A simple curve $\gamma$ in $\A$ connecting $B_0$ and $B_1$, is called a \textit{Brouwer curve} for $f$ if 
$$\gamma\cap f(\gamma)=\emptyset.$$ 
A Brouwer curve $\gamma$ is called \textit{smooth} if $\gamma = \phi([0,1])$ where $\phi$ is 
$C^\infty$-smooth and regular, meaning that the derivative is nowhere vanishing.   Brouwer curves for homeomorphisms of $\tilde\A$ are defined similarly.		

\end{defi}
Note that if $F$ is a lift of a homeomorphism $f:\A\to\A$ and $\tilde\gamma$ is a lift of a Brouwer curve $\gamma\subset\A$ for $f$, then 
\aryst
T^{k}F(\tilde \gamma) \cap \tilde\gamma = \emptyset \qquad \forall k \in \Z.  
\earyst
\begin{defi}[$Q$-good Brouwer curves]\label{D:good Brouwer}
	Let $f$ be a homeomorphism of $\A$ and $Q\geq 2$ an integer.  A Brouwer curve $\gamma$ will be called \textit{$Q$-good} if $\gamma$ is also a Brouwer curve for each of the maps $f,f^2,\cdots,f^{Q-1}$, equivalently if $\gamma,f(\gamma),\ldots,f^{Q-1}(\gamma)$ are pair-wise disjoint.
\end{defi}

\begin{defi}
	Let $\gamma_1, \gamma_2$ be two disjoint simple curves in $\A$ connecting $B_0$ and $B_1$.    There is a unique closed region $\cR$ 
	in $\A$ with left side $\gamma_1$ and right side $\gamma_2$.   More precisely, if each $\gamma_i$ is oriented from $B_0$ to $B_1$, then $\partial\cR\cap (\A\backslash\partial\A)$ has orientation agreeing with $\gamma_2 - \gamma_1$.  We say that \textit{$\cR$ is the region bounded by $(\gamma_1, \gamma_2)$}.
\end{defi}

\begin{defi}[Admissible coordinates]\label{def ad coord} 
	Let $\gamma$ be a Brouwer curve for $f$, and let $\cR \subset \A$ be the  closed region  bounded by $(\gamma, f(\gamma))$. We say that an orientation  	preserving $C^\infty$ -diffeomorphism 
	$$\phi : \cU \to \cR'$$ 
	from  an open neighborhood $\cU\subset \tilde{\A}$ of $[0,1]^2$ to an open neighborhood $\cR'\subset\A$ of $\cR$ is 
	an \textit{admissible coordinate} for $(\cR, f)$ if the following hold: 
	\enmt 
	\item $\phi$ has constant Jacobian,
	\item $\phi$ satisfies 
	\aryst 
	\phi(\{0\} \times [0,1]) &=& \gamma, \label{item 1} \\
	\phi(\{1\} \times [0,1]) &=& f(\gamma), \\
	\phi([0,1] \times \{0\}) &\subset& B_0, \\
	\phi([0,1] \times \{1\}) &\subset& B_1, \label{item 4}
	\earyst 
	\item there is a neighborhood $\cU_L$ of $\{0\} \times [0,1]$ in $\cU$, so that
	\aryst
	f \phi(x) = \phi T(x)  \qquad \forall x \in \cU_L.
	\earyst
	\eenmt
	Without loss of generality $T(\cU_L) \subset \cU$ by choosing $\cU_L$ sufficiently small.
\end{defi}
We make a similar definition for lifts:

\begin{defi}\label{def ad coord lift} 
Let $\tilde \gamma$ be a lift of a Brouwer curve $\gamma$ for $f$, 
let $F$ be a lift of $f$ such that 
$F(\tilde \gamma)$ is on the right of $\tilde \gamma$,
and let $\tilde \cR$ be the region bounded by $(\tilde \gamma, F(\tilde \gamma))$. \clb We say that a $C^\infty$-diffeomorphism 
$$\phi: \cU \to \cR'$$ 
from an open neighborhood $\cU$ of $[0,1]^2$ in $\tilde\A$ to an open neighborhood $\cR'$ of $\tilde \cR$ in $\tilde\A$ is an \textit{admissible coordinate for $(\tilde \cR, F)$} if $\phi$ satisfies the analogous properties to items 1-3 in Definition \ref{def ad coord}, with $(\cR, \A, f)$ replaced by $(\tilde \cR, \tilde \A, F)$.
\end{defi}
\begin{rema} \label{rema unique continuation}
	By item $(3)$ in Definition \ref{def ad coord}, an admissible coordinate is determined by its restriction to $[0,1]^2$ in the following sense: 
	If two admissible coordinates for $(\cR, f)$ agree on $[0,1]^2$ then they agree on some open neighborhood of $[0,1]^2$ in $\tilde{\A}$.   
	For this reason, we will often identify an admissible coordinate for $(\cR, f)$ with a map  from $[0,1]^2$ to $\cR$.  Analogously for lifts.    
\end{rema}

%
%
%
By the next Lemma admissible coordinates always exist. 
	 \begin{lemma}\label{A:admissible-coordinates-modified}
	 Let $r \geq 1$ and $\theta \in (0,1)$.    Suppose $g\in\diff^{r+1,\theta}(\A, \omega)$ has a smooth Brouwer curve $\gamma$, 
	 parameterised by arclength, that meets the boundary of $\A$ orthogonally near both end points.   Then the region $\cR$ bounded by   
	 $(\gamma,\,g(\gamma))$ has admissible coordinates $\phi:[0,1]^2\to\cR$ so that 
	 \aryst
	 \|\phi\|_{\diff^{r,\theta}([0,1]^2)}\leq D_r(\|g\|_{\diff^{r+1,\theta}(\A)},\|\gamma\|_{C^{r+3,\theta}},  \omega(\cR)^{-1}), 
	 \earyst
	 where $D_r$ is some continuous increasing function, independent of $g$ and $\gamma$, and $\omega(\cR)$ is 
	 the $\omega$-area of $\cR$. 
	\end{lemma} 
In particular, there is a neighborhood $\cV$ of $g$ in $\diff^{r+1,\theta}(\A, \omega)$ such that for every  $g' \in \cV$, 
		the region $(\gamma,\,g'(\gamma))$ has admissible coordinates $\phi'$ for which $\|\phi'\|_{\diff^{r,\theta}([0,1]^2)}$ is bounded by a constant depending only on $\|g\|_{\diff^{r+1,\theta}}$, $\|\gamma\|_{C^{r+3,\theta}}$ and $r$.  
%
%
\begin{proof}
	 We construct the admissible coordinates $\phi:[0,1]^2\to\cR$ in stages, working inwards from the four edges.  
	 
	 The first step is to construct a model for $\phi$ on an open neighborhood of the left edge $\{0\} \times  [0,1]$, which we will call $\psi_L$.     
	 Indeed, as $\gamma:[0,1]\rightarrow\A$ 
	 is parametrised by arclength, we have $\|\dot{\gamma}\|=L$ is constant for some $L\geq 1$.   
	 Thus $n:=-i\dot{\gamma}/L$ is a unit normal vector field along 
	 $\gamma$ and the map $:\tilde{A}\rightarrow\R^2,$ $(x,y)\mapsto \gamma(y)+x\omega(\cR)n(y)/L$ extends $\gamma$ to a $C^\infty$-smooth 
	 diffeomorphism $\varphi$ from a small tubular neighborhood of $\{0\} \times  [0,1]$ in $\tilde\A$ 
	 to a neighborhood of the image of $\gamma$ in $\A$, with constant Jacobian $\omega(\cR)$ {\it along}, i.e.\ at each point of,  $\{0\} \times  [0,1]$.  
	 Since $\gamma$ meets the boundary of $\A$ orthogonally near both end points, we also see that $\varphi$ has constant Jacobian $\omega(\cR)$ 
	 on a neighborhood of the end points $(0,0), (0,1)$ within $\tilde\A$.   A computation shows  
	 \ary\label{eq admissible 1}
					\|\varphi\|_{\diff^{r+2,\theta}}\leq 2c\|\gamma\|_{C^{r+3,\theta}}.  
	\eary
	where $c:=\max\{\omega(\cR)/L,\,L/\omega(\cR)\}$.   
We need a further argument to make the Jacobian constant also {\it near} to $\{0\} \times  [0,1]$, i.e.\ on a neighborhood in $\tilde\A$.     To do this, denote by $\omega(\cR)\omega_1$ the pull back of $\omega$ by $\varphi$.  Thus $\omega_1$ is a $C^\infty$-smooth symplectic form defined on a tubular neighborhood of $\{0\} \times  [0,1]$ in $\tilde\A$, which coincides with $\omega$ at each point of $\{0\} \times  [0,1]$ and also on small neighborhoods of $(0,0)$ and $(0,1)$ within 
$\tilde\A$.   
	By a Moser deformation argument following the proof of Lemma 3.14 in \cite{MS}, one finds a coordinate change on a tubular neighborhood of $\{0\} \times  [0,1]$ which is the identity on $\{0\} \times  [0,1]$, and also the identity on small neighborhoods of $(0,0)$ and $(0,1)$, which pulls back $\omega_1$ to $\omega$. 
	Thus, precomposing $\varphi$ with this $C^\infty$-smooth coordinate change we obtain a map, which we denote 
	by $\psi_L$, 
	which maps a neighborhood of $\{0\} \times  [0,1]$ in $\tilde\A$ diffeomorphically to a neighborhood of $\gamma$, it maps $\{0\} \times  [0,1]$ precisely to $\gamma$, and has constant Jacobian equal to $\omega(\cR)$.    
	The deformation argument used in \cite{MS} however ''loses a derivative''\footnote{This is in contrast to some of the results in 
	\cite{DacMos,A10,T}, which however do not immediately apply; \cite{DacMos} 
	requires the boundary to be smooth, while \cite{A10,T} requires constant Jacobian near to the boundary -  
	something we have not yet achieved.}.   
	Thus although $\psi_L$ is a $C^\infty$-smooth diffeomorphism, only $\|\psi_L\|_{\diff^{r+1,\theta}}$ is controlled by $\|\varphi\|_{\diff^{r+2,\theta}}$, i.e.\ by $\|\gamma\|_{C^{r+3,\theta}}$, 
	see \eqref{eq admissible 1}.  	
	 	 
	 Now we set $\psi_R:=g \psi_L T^{-1}$, which defines a $C^{r+1,\theta}$-diffeomorphism with constant Jacobian $\omega(\cR)$ from a neighborhood of $\{1\} \times  [0,1]$ in $\tilde\A$ onto its image, a neighborhood of $g(\gamma)$ in $\A$.    Clearly also $\psi_R$ maps $\{1\} \times  [0,1]$ precisely onto $g(\gamma)$.   
	 Note that $\|\psi_R\|_{\diff^{r+1,\theta}}$ is controlled by $\|g\|_{\diff^{r+1,\theta}}$ and $\|\psi_L\|_{\diff^{r+1,\theta}}$, and hence by $\|g\|_{\diff^{r+1,\theta}}$ and $\|\gamma\|_{C^{r+3,\theta}}$.  
	 
	 We construct the chart $\phi: [0,1]^2 \to \A$ by first specifying its restriction to a neighborhood of $\{0\} \times [0,1]$, resp.\ $\{1\} \times [0,1]$, to be $\psi_L$, resp.\ $\psi_R$.  
	  (Thus $\phi$ satisfies item 3 in Definition \ref{def ad coord}.)    Then we extend $\phi$ to all of $[0,1]^2$ in a more or less arbitrary way, but so that $\phi$ maps $[0,1]^2$ diffeomorphically to $\cR$ (still in the class $C^{r+1,\theta}$).   We also arrange that {\it along} the top and bottom sides, $[0,1] \times \{1\}$ and $[0,1] \times \{0\}$, $\phi$ has constant Jacobian $\omega(\cR)$.    This condition is easy to achieve   by hand.   Now $\phi$ is a $C^{r+1,\theta}$-diffeomorphism and $\|\phi\|_{\diff^{r+1,\theta}}$ is controlled by $\|g\|_{\diff^{r+1,\theta}}$ and $\|\gamma\|_{C^{r+3,\theta}}$.
	 
	 Now we make the Jacobian of $\phi$ constant also {\it near} to the top and bottom $[0,1] \times \{1\}$ and $[0,1] \times \{0\}$, without altering it  
	 near the left and right edges $\{0\} \times  [0,1]$ and $\{1\} \times  [0,1]$.    
		 To do this, set $\phi^*\omega=\omega(\cR)\omega_2$.  
		 Thus $\omega_2$ is a smooth symplectic form on $[0,1]^2$, which coincides with $\omega$ {\it near} to the left and right edges, and 
	 also {\it on} the top and bottom edges.   
	Applying the proof of Lemma 3.14 in \cite{MS} again, this time to the top and bottom edges, we find a further coordinate change on collar neighborhoods of $[0,1] \times \{1\}$ and $[0,1] \times \{0\}$, which is the identity {\it on} these two edges, is also the identity {\it near} to the corners $(0,1), (1,1), (0,0), (1,0)$, and which pulls back $\omega_2$ to $\omega$ (on the collar neighborhoods).   
	We extend this coordinate change from the collar neighborhoods of the top and bottom edges to 
	all of $[0,1]^2$, in such a way that it remains the identity map {\it near} to the left 
	and right edges of $[0,1]^2$.   
	Precomposing $\phi$ with this coordinate change, we may assume that $\phi:[0,1]^2\to\cR$ is a diffeomorphism that has constant Jacobian $\omega(\cR)$ 
	near to all four edges of $[0,1]^2$, and that it still coincides with $\psi_L$, resp.\ $\psi_R$, on a collar neighborhood of $\{0\} \times [0,1]$, resp.\ $\{1\} \times [0,1]$.    
		Since we used the Moser argument from \cite{MS} we lost a derivative of control, so now $\phi$ is just a $C^{r,\theta}$-diffeomorphism 
		with $\|\phi\|_{\diff^{r,\theta}}$  controlled by $\|g\|_{\diff^{r+1,\theta}}$ and $\|\gamma\|_{C^{r+3,\theta}}$. 
	
Finally, we modify $\phi$ away from a neighborhood of the boundary, so as to have constant Jacobian on all of $[0,1]^2$.    
This can be done using either Theorem 5 from \cite{DacMos}, or \cite{T}. 
In either case, one replaces $[0,1]^2$ by a subset $\Omega$ with smooth boundary, where $\partial\Omega$ lies sufficiently close to $\partial [0,1]^2$ in the interior of the region where the Jacobian is already constant.    If we use \cite{T} this final coordinate change can be done without losing a derivative, so that $\phi$ is still a $C^{r,\theta}$-diffeomorphism with $\|\phi\|_{\diff^{r,\theta}}$ controlled by $\|g\|_{\diff^{r+1,\theta}}$ and $\|\gamma\|_{C^{r+3,\theta}}$.
The resulting $\phi$ satisfies all the items in Definition \ref{def ad coord}.   
\end{proof}

\subsection{Pseudo-rotations}
The name pseudo-rotation was introduced in \cite{BCLP}, where it referred to any non-wandering homeomorphism $f : \A \to \A$ isotopic to 
the identity with no periodic points.    In this paper we use the following slightly more restricted definition.   Recall from \ref{SS:lifts} that $\omega$ is the standard smooth area form on $\A$:  

\begin{defi}\label{defi prot}
A \textit{pseudo-rotation} is an $\omega$-preserving diffeomorphism $f : \A \to \A$ that is isotopic to the identity and has no periodic points. 
\end{defi}
Unless otherwise specified, pseudo-rotations will always be $C^\infty$-smooth.   A $C^0$-pseudo-rotation refers to the case that $f$ is a homeomorphism.  

It is a non-trivial fact that all the trajectories of a pseudo-rotation have the same asymptotic rotation number.   
More precisely: 
Let $p_1 : \tilde \A \to \R$ be projection onto the first coordinate.  
	By results of Franks and Franks-Handel \cite{F88, FH12}, if $F:\tilde\A\rightarrow\tilde\A$ is any lift of a pseudo-rotation $f$ then the limit  
	\ary\label{def rotation number F}
	\rho(F) := \lim_{n \to \infty} \frac{1}{n}\Big(p_1(F^{n}(\tilde x)) - p_1(\tilde x) \Big)\in\R 
	\eary
	exists for all $\tilde x\in\tilde\A$, is independent of $\tilde x$, and is irrational.     
	Therefore we can define: 

\begin{defi}\label{def rotation number}
	Let $f:\A \to \A$ be a pseudo-rotation.   Then for each lift $F:\tilde\A\rightarrow\tilde\A$ we call the limit $\rho(F)\in \R\backslash\Q$ above the \textit{rotation number} of $F$ 
	and we call the fractional part 
	\aryst
	\rho(f) := \{\rho(F)\}\in(0,1)\backslash\Q
	\earyst
	the \textit{rotation number} of $f$.    
\end{defi}

Note that $\{\rho(F)\}$ is independent of the lift $F$ because evidently $\rho(T^jF)=\rho(F)+j$ from the definition.   
Note also that for any pseudo-rotation $f$ there exists a unique lift $F$ for which $\rho(F) =\rho(f)$.   

\section{Existence of a Brouwer curve with uniform bounds}\label{S:Brouwer}
By work of Guilliou \cite{G} every pseudo-rotation $f$ has a Brouwer curve $\gamma$.     In this section, under restrictions on the 
rotation number, we obtain a $q_1$-good Brouwer curve with bounds on the geometric degeneracy of an associated fundamental region, see Lemma \ref{lem brouwer curve compactness} and Proposition \ref{prop. main}.   

In this section we exploit smallness assumptions on the rotation number of $f$ and on the rotation number of $f^{q_1}$, i.e.\ on $\a=\rho(f)$ and on $\G(\a)$ relative to $\a$.   
These conditions are of course very restrictive, but in Section \ref{S:renorm} we expand the applicability of these results using a renormalization scheme.   

We begin with the following fundamental estimates, which are essentially proven in \cite{A}. 

\begin{thm}\label{thm aflxz}
	There is a sequence of continuous, increasing functions $A_r : (0, 1/2) \times  \R_+ \to \R_+$, over $r\in\N_0$, 
	with $\lim_{t \to 0}A_r(t,\, \cdot\,) \equiv 0$ and $\lim_{K \to +\infty}A_r(\cdot, K) \equiv +\infty$, so that the following holds.   If $f:\A\to\A$ is a 
	pseudo-rotation, then 
	\ary
	d_{C^0(\A)}(f,{\rm Id}) &<& A_0(\norm{\rho(f)}_{\R / \Z},  \|Df\|_{\A}),\label{aflxz r=0}\\
	\norm{D^rf}_{\A} &<&  A_r(\norm{\rho(f)}_{\R / \Z}, \norm{D^{r+1}f}_{\A}), \quad \forall r \geq 1.\label{aflxz r>0}
	\eary
	If $f$ is only a $C^{r_0}$-smooth pseudo-rotation, with $r_0\in\N$, then \eqref{aflxz r=0} is valid, and if $r_0\geq 2$ then \eqref{aflxz r>0} holds for $1\leq r\leq r_0-1$.   
\end{thm}
\begin{proof}
	We abrieviate $K_r:=\norm{D^rf}_{\A}$ for each $r\in\N$. 
	Lemma 2.1 in \cite{A} is a symplectic statement about $C^0$-pseudo-rotations on the disc $\D$, so it transfers directly to pseudo-rotations on the annulus and 
	we conclude the following for $f$: Every topological disc $D\subset\A$ with area greater than $\norm{\rho(f)}_{\R / \Z}$ has $f(D)\cap D\neq\emptyset$.   
	So if $f$ is at least $C^1$, arguing analogously to \cite[Corollary B]{A},  	
	\aryst
	d_{C^0(\A)}(f,{\rm Id}) < (1+2K_1)\norm{\rho(f)}_{\R / \Z}^{1/2} =: A_0(\norm{\rho(f)}_{\R / \Z},K_1).	
	\earyst
	Lifting the statement about non-displaceable discs to the universal covering gives; for every topological disc $D\subset\tilde\A$ with 
	area greater than $\norm{\rho(f)}_{\R / \Z}$ and which is the lift of a topological disc in $\A$, there is a lift $F$ of $f$ (depending on $D$) 
	for which $F(D)\cap D\neq\emptyset$.   This leads to a bound of the form $\|F-\id_{\tilde\A}\|_{D}\leq (1+\|DF\|)\norm{\rho(f)}_{\R / \Z}^{1/2}\leq(1+2K_1)\norm{\rho(f)}_{\R / \Z}^{1/2}$.	Suppose $f$ is of class $C^{r_0}$, $r_0\geq2$.    
	Applying the convexity estimates that interpolate between $C^s$-norms, to $F:D\to\R^2$, we get for each integer $1\leq r\leq r_0-1$: 
	\aryst
	\norm{D^r F}_{D} &\leq& c_r\, \|F-\id\|_{D}^{\frac{1}{r+1}}\norm{F-\id}_{C^{r+1}(D)}^{\frac{r}{r+1}} \\
	&\leq&c_r \norm{\rho(f)}_{\R / \Z}^{\frac{1}{2(r+1)}}(1+2K_{r+1}) =: A_r( \norm{\rho(f)}_{\R / \Z},  K_{r+1})
	\earyst
for some constant $c_r>0$.   Since we can find such a $D\subset\tilde\A$ containing an arbitrary point, we conclude $\norm{D^r f}_{\A}=\norm{D^r F}_{\tilde\A}\leq A_r( \norm{\rho(f)}_{\R / \Z},  K_{r+1}).$   
\end{proof}
By Theorem \ref{thm aflxz} if the rotation number of a pseudo-rotation $f$ is sufficiently small, then $d_{C^0(\A)}(f,{\rm Id}_{\A})  <1/2$ and hence by Lemma \ref{lem:lift} 
\ary\label{lem Fidfid}
			\norm{F - {\rm Id}_{\tilde\A}} = d_{C^0(\A)}(f,{\rm Id}_{\A})  <1/2
\eary	 
for some unique lift $F$.   
More precisely, there is an increasing function 
\ary\label{def a}
\varepsilon_0 :\R_+ \to (0,1/2)
\eary 
so that \eqref{lem Fidfid} holds whenever 
\aryst
\|\rho(f)\|_{\R / \Z} < \varepsilon_0(\norm{Df}^{-1}_{\A}).  
\earyst	
The next Lemma says that under smallness conditions on both the rotation number $\a$ and on $\G(\a)$, a pseudo-rotation moves each point 
a uniform positive amount.    
\begin{lemma}\label{lem no almost fixed point}
		There are increasing functions 
		$\delta, \varepsilon_1 : (0,1)\times(0,1) \to \R_+$ 
		with $\lim_{t \to 0} \varepsilon_1(t,\cdot)\equiv 0$, so that the following holds.   If $\alpha\in (0,1)\backslash\Q$ satisfies  
		\ary
		\a &<& \varepsilon_0(K^{-1}), \label{eq no fixed 02} \\
			\G(\a) &<& \varepsilon_1(\a, K^{-1}) \label{eq no fixed 01}
		\eary
		for some $K>1$, 
		where $\varepsilon_0$ is given in \eqref{def a}, then every $C^{1}$-smooth pseudo-rotation with $\norm{Df}_{\A} \leq K$ 
		and $\rho(f)=\a$ satisfies   
		\ary\label{min distance}
		\inf_{x \in \A} d_{\A}(x, f(x)) \geq \delta(\a, K^{-1} ).
		\eary
		Moreover, without loss of generality 
		\ary \label{eq c''issmall}
		\varepsilon_1(\alpha, K^{-1}) < \varepsilon_0(K^{-\alpha^{-1}}) 
		\eary
		for all $(\alpha,K) \in (0,1)\times(1,\infty)$.  
	\end{lemma}
	Note that the conditions \eqref{eq no fixed 02} and \eqref{eq no fixed 01} are easily fulfilled:  For all $\epsilon_0,\epsilon_1>0$ 
	the set of $\a\in(0,1)\backslash\Q$ with $\a<\epsilon_0$ and $\G(\a)<\epsilon_1$ is non-empty.   

	\begin{proof}
		Define $\varepsilon_1:(0,1)\times(0,1) \to \R_+$ as follows: $\varepsilon_1(t,K^{-1})$ to be the minimum of: 
		\ary\label{D:C''}
							\sup\Big\{ s\in(0,1)\ |\ A_0(s, K^{1/t})\leq 1/4 \Big\},\qquad\frac{1}{2}\varepsilon_0(K^{-1/t}).
		\eary
		Since $\lim_{s \to 0}A_0(s,\cdot)= 0$ the supremum is over a non-empty set, so that the first value is well-defined and strictly positive.   
		Clearly \eqref{eq c''issmall} holds due to the value on the right in \eqref{D:C''}.   Since $A_0$ and $\varepsilon_0$ are increasing, one 
		checks that $\varepsilon_1$ so defined is also increasing.   
		Moreover, for each fixed $K>1$, we have $\lim_{t \to 0}A_0(\cdot, K^{1/t})=+\infty$, so that the first value 
		in \eqref{D:C''} tends to zero as $t$ tends to zero.   Thus $\lim_{t \to 0}\varepsilon_1(t,\cdot)\equiv 0$.   Finally, since $\varepsilon_1(t,K^{-1})$ 
		is at most the first value in \eqref{D:C''}, by continuity of $A_0$ we have 
		\ary\label{eq no fixed 1}
				\qquad\qquad	A_0\big(\varepsilon_1(t,K^{-1}), K^{1/t}\big)\leq1/4<1/2\qquad\forall (t,K)\in(0,1)\times(1,\infty).
		\eary
		Now define the increasing function $\delta:  (0,1) \times(0,1) \to \R_+$ by 
		\aryst 
						\delta(t,K^{-1}) := \frac{1}{2}\frac{K-1}{K^{t^{-1}}-1}.    
		\earyst
		We will show that the Lemma holds for these two functions. 
		Let $K\geq 1$ be fixed.  Suppose $\a\in(0,1)\backslash\Q$ satisfies (\ref{eq no fixed 02}) and (\ref{eq no fixed 01}).   Let $f$ be a $C^1$-smooth 
		pseudo-rotation with $\norm{Df}_{\A} < K$ and $\rho(f)=\a$.   
		Let $F$ be the lift of $f$ with $\rho(F) = \rho(f)$ and abbrieviate $q := q(\a)=\lfloor 1/\a\rfloor\leq 1/\a.$   
		By Theorem \ref{thm aflxz}, 
		\aryst
		d_{C^0(\A)}(f^q,{\rm Id}) \leq  A_0(\|\rho(f^{q})\|_{\R/\Z}, \|Df^{q}\|)
									\leq  A_0(\|q\a\|_{\R/\Z}, K^{q}).\label{bound on f^q}
		\earyst
		From \eqref{E:pn distance} and \eqref{eq no fixed 01} we have $\|q\a\|_{\R/\Z}=p_1-q_1\a=1-q\a=\a\G(\a)<\G(\a)< \varepsilon_1(\a, K^{-1})$.    
		Thus by monotonicity of $A_0$, the line above, and \eqref{eq no fixed 1} yield 
		$d_{C^0(\A)}(f^q,{\rm Id})\leq  A_0(\varepsilon_1(\a, K^{-1}), K^{1/\a})< 1/2$, so by Lemma \ref{lem:lift}
		\aryst 
		\|F^{q}-T\|_{\tilde\A}= \|T^{-1}F^{q}-{\rm{Id}}\|_{\tilde\A}=d_{C^0(\A)}(f^q,{\rm Id})<1/2,  
		\earyst
		because $T^{-1}F^q$ is 
		the unique lift of $f^q$ with rotation number in $(-1/2, 1/2)$ (since 
		$\rho(T^{-1}F^q)=q\rho(F)-1=q\a-1\in(-1/2,1/2)$).   Thus for all $\tilde x\in\tilde\A$
		\ary \label{eq no fixed 4}
		\qquad |F^{q}(\tilde x)-\tilde x|\geq |T(\tilde x)-\tilde x| - |F^{q}(\tilde x)-T(\tilde x)| > 1-1/2 = 1/2.  
		\eary
	   On the other hand 
		$|F^q(\tilde x)-\tilde x|  \leq \sum_{j=0}^{q-1}\|DF\|_{[0,1]^2}^j|F(\tilde x)-\tilde x| \leq C|F(\tilde x)-\tilde x|$ 
		where $C=\frac{K^{q}-1}{K-1}\leq \frac{K^{\a^{-1}}-1}{K-1}$.  
		Combining with (\ref{eq no fixed 4}) yields 
		\aryst
					|F(\tilde x)-\tilde x| \geq \frac{1}{2C}\qquad \forall \tilde x\in\tilde\A.    
		\earyst
		The right hand side is $\delta(\a,K^{-1})$ by definition.  
	    	By Lemma \ref{lem:lift} together with \eqref{eq no fixed 02} and \eqref{lem Fidfid}, the left hand side is $d_{\A}(f(x),x)$, 
		for $x=\pi(\tilde x)$.   
		\end{proof}
%
%
%
%
\begin{lemma}\label{lem brouwer curve compactness}
	There is a sequence of  increasing functions  
	$C_r : \R_+^{2} \to \R_+$ over $r\geq 1$,  
	so that  the following holds.   Let $f$ be a pseudo-rotation with $\rho(f)=\a$, with
	\aryst
		\a &< &  \varepsilon_0(\|Df\|^{-1}), \\
		\cG(\a) &<& \varepsilon_1(\a, \|Df\|^{-1})
		\earyst
		where $\varepsilon_0$ is from \eqref{def a} and $\varepsilon_1$ is from Lemma \ref{lem no almost fixed point}.
	Then for each $r\in\N$ and $\theta\in(0,1)$, $f$ has a smooth Brouwer curve $\gamma=\gamma_r$ for which 
	the closed region bounded by $(\gamma, f(\gamma))$ has admissible coordinates $\phi=\phi_r:[0,1]^2\to(\gamma, f(\gamma))$ with 
	\ary\label{L:phi estimates}
				\|\phi\|_{\diff^{r,\theta}([0,1]^2)} \leq C_r(\a^{-1}, \|f\|_{\diff^{r+2}(\A)}). 
	\eary
\end{lemma}
\begin{proof}
	Fix $r\geq 1$ and $\theta\in(0,1)$.   
	For each $c>0$ and $K\geq 1$ define the following subset of $\diff^{r+2}(\A, \omega)$:
	\aryst
 			\cH(c,K) := \Big\{g \in \diff^{r+2}(\A, \omega)\ \Big|\ \norm{g}_{\diff^{r+2}(\A)} \leq K,\quad \inf_{x \in \A} d_{\A}(x, g(x)) \geq c \Big\}.
	\earyst
Denote by  
\aryst
	\overline{\cH(c,K)} \subset \diff^{r+1,\theta}(\A, \omega)
\earyst
the closure of $\cH(c,K)$ in the $\diff^{r+1,\theta}$-topology.   We observe that $\overline{\cH(c,K)}$ is compact in $\diff^{r+1,\theta}(\A)$. 
(Even though zeroth order terms are not included in $\|\cdot\|_{\diff^r}$, they are trivially uniformly bounded.)    
Clearly $\overline{\cH(c,K)}$ contains only diffeomorphisms without fixed points.   
Moreover, for $c'\geq c$ and $K'\leq K$ we have $\overline{\cH(c',K')}\subset \overline{\cH(c,K)}$.   
We now prove a version of the Lemma for elements of $\overline{\cH(c,K)}$.   Then we argue that the union of 
$\overline{\cH(c,K)}$ over $c\in(0,1]$ contains all pseudo-rotations satisfying the assumptions of the Lemma.   

	Fix $c>0, K\geq1$ and consider $g \in \overline{\cH(c,K)}$.   Since $g$ has no fixed points, a strong refinement of Brouwer's plane 
	translation theorem due to Guillou \cite[Th\'eor\`em 5.1]{G}  
		yields a $C^0$ Brouwer curve $\gamma_0$ 
	for $g$ in the sense of Definition \ref{def:Brouwer}.  
	Any sufficiently $C^0$-close smoothly embedded 
	approximation of $\gamma_0$ that continues to connect the two boundary components yields a smooth Brouwer curve $\gamma$ for $g$.   
	Clearly we can choose $\gamma$ to meet the boundary components orthogonally and parameterised by 
	arclength. 
	
	Applying Lemma \ref{A:admissible-coordinates-modified} gives an open neighborhood $\cV_g\subset\diff^{r+1,\theta}(\A)$ of $g$, 
	so that the following holds:  For each $g'\in\cV_g$, $\gamma$ is still a Brouwer curve and the region between $\gamma$ and $g'(\gamma)$ 
	has admissible coordinates on which $\|\cdot\|_{\diff^{r,\theta}([0,1]^2)}$ is uniformly bounded throughout $\cV_g$.      
	By compactness of $\overline{\cH(c,K)}\subset \diff^{r+1,\theta}(\A, \omega)$ the collection of all such 
	neighborhoods $\cV_g$ as $g$ varies over $\overline{\cH(c,K)}$ has a finite subcover of $\overline{\cH(c,K)}$.    
	 Thus we obtain a uniform bound $E_r(c,K)>0$ on the $\|\cdot\|_{\diff^{r,\theta}([0,1]^2)}$-size of admissible coordinates that  
	applies to all elements of $\overline{\cH(c,K)}$.   
	From the inclusions $\overline{\cH(c',K')}\subset \overline{\cH(c,K)}$ for $c'>c, K'<K$, we can assume $E_r(c,K)$ is 
	decreasing in $c$ and increasing in $K$.   
	
	Now, suppose $f$ is as in the Lemma.   
	Then by Lemma \ref{lem no almost fixed point} $f\in \cH(c,K)$ where $c=\delta(\a, \|Df\|^{-1})$ and $K=\|f\|_{\diff^{r+2}(\A)}$.   
Thus there exists a region $(\gamma,\,f(\gamma))$, for some smooth Brouwer curve $\gamma$, having admissible coordinates $\phi$ with 
$\|\phi\|_{\diff^{r,\theta}([0,1]^2)}\leq E_r(\delta(\a, \|Df\|^{-1}), \|f\|_{\diff^{r+2}(\A)})$.  So Lemma \ref{lem brouwer curve compactness} 
	holds with $C_r(\a^{-1},K):=E_r(\delta(\a, K^{-1}),K)$.  
\end{proof}
With further restrictions on the rotation number the Brouwer curve in Lemma \ref{lem brouwer curve compactness} is $q_1$-good:    
\begin{prop}\label{prop. main}
	There is an increasing function $\varepsilon_2 : \R_+^{2} \to \R_+$ with $\varepsilon_2\leq \varepsilon_1$, 
	so that  the following holds.   Let $f$ be a pseudo-rotation with $\rho(f)=\a$ so that    
	\ary
		\a &< &  \varepsilon_0(\|Df\|^{-1}),\nonumber \\
		\cG(\a) &<& \varepsilon_2(\a, \|f\|_{\diff^3(\A)}^{-1}).\label{E:assumption on G(a)}
	\eary
	Then each Brouwer curve $\gamma=\gamma_r$ for $f$ produced by Lemma \ref{lem brouwer curve compactness} (for any $r\in\N$) is $q(\a)$-good.  
\end{prop}
\begin{proof}
Define $\varepsilon_2:\R_+^{2} \to \R_+$ as follows: $\varepsilon_2(\a,K^{-1})$ should be the minimum of 
\aryst
		\sup\Big\{ s\in(0,1/2)\ \Big|\ A_0(s, K^{q(\a)})\leq\frac{1}{2K^{q(\a)}}C_1(\a^{-1},K)^{-1}\Big\}, \quad \varepsilon_1(\a,K^{-1}).  
\earyst 
Since $\lim_{s\to 0}A_0(s,K^{q(\a)})=0$ the supremum is over a non-empty set and hence well-defined and strictly positive.   Clearly $\varepsilon_2\leq \varepsilon_1$.   
One checks that $\varepsilon_2$ is increasing because $A_0, C_1$ and $\varepsilon_1$ are increasing.   Abrieviate $K_1:=\|f\|_{\diff^{1}(\A)}$ and $K_3:=\|f\|_{\diff^{3}(\A)}$.  
From \eqref{E:assumption on G(a)} 
\ary\label{E:condition G(a)} 
						A_0(\G(\a), K^{q(\a)}_{3})<\frac{1}{K^{q(\a)}_{3}}C_1(\a^{-1},K_{3})^{-1}.  
\eary
 Let $\gamma$ be the Brouwer curve given by Lemma \ref{lem brouwer curve compactness} for $r=r_0$, and let $\tilde\gamma\subset\tilde\A$ be any lift of $\gamma$.    
Let $F$ be the lift of $f$ with $\rho(F) = \a$.   
Since $\tilde\gamma\cap F(\tilde \gamma)=\emptyset$ 
and $\rho(F) > 0$ we have $\tilde\gamma<F(\tilde \gamma)$ from considering boundary points. 
Using injectivity of $F$ and the order of boundary points we conclude successively 
$\tilde\gamma<F(\tilde \gamma)<F^2(\tilde\gamma)< \cdots < F^{i}(\tilde \gamma)$ 
for all $i\in\N$.  Thus it suffices to show that 
\ary
\label{eq rightbound}
		F^{q(\a)-1}(\tilde \gamma)<T\tilde\gamma.  
\eary
Indeed, then $F(\tilde \gamma),\ldots, F^{q(\a)-1}(\tilde \gamma)$ will all lie strictly in the region between $\tilde \gamma$ and $T(\tilde \gamma)$ 
so that $f(\gamma),\ldots, f^{q(\a)-1}(\gamma)$ will be disjoint from $\gamma$ as required.   
Arguing indirectly, suppose (\ref{eq rightbound}) is not true.   Then  
\ary
\label{eq Qgood intersection}
		F^{q(\a)-1}(\tilde \gamma)\cap T\tilde\gamma\neq\emptyset
\eary
because of the order of boundary points.   
We will now show that this implies $F^{q(\a)-1}(\tilde \gamma)$ and $F^{q(\a)}(\tilde \gamma)$ pass somewhere 
close to each other, because $F^{q(\a)}(\tilde \gamma)$ is close to $T\tilde\gamma$.   
Indeed, from \eqref{E:assumption on G(a)} and $\varepsilon_2\leq \varepsilon_1$ we have 
$\|\rho(f^{q(\a)})\|_{\R/\Z} = |q(\a)\a-1| =\a\cG(\a)<\cG(\a)\leq \varepsilon_2(\a,K_{3}^{-1})\leq \varepsilon_1(\a,K_{3}^{-1})$    
so by \eqref{eq c''issmall}, 
$\|\rho(f^{q(\a)})\|_{\R/\Z}< \varepsilon_0(K_{3}^{-\a^{-1}}) \leq \varepsilon_0(\| Df\|^{-q(\a)}) \leq \varepsilon_0(\| Df^{q(\a)}\|^{-1})$.   
Thus we can apply \eqref{lem Fidfid} to $f^{q(\a)}$ and conclude, 
\aryst
d_{H}( F^{q(\a)}(\tilde \gamma), T\tilde \gamma ) \leq   \| F^{q(\a)}- T\|  =   \| T^{-1}F^{q(\a)}- {\rm Id}\|  =  d_{C^0(\A)}(f^{q(\a)},{\rm Id})
\earyst
for the Hausdorff distance $d_H$ on compact subsets of $\tilde\A$.   So by Theorem \ref{thm aflxz}, 
\begin{equation}\label{eq Qgood Hausdorff}
	d_{H}( F^{q(\a)}(\tilde \gamma), T\tilde \gamma ) \leq  A_0(\|\rho(f^{q(\a)})\|_{\R/\Z},K_1^{q(\a)}) 
											\leq A_0(\cG(\a), K_{1}^{q(\a)}).
\end{equation}
From (\ref{eq Qgood intersection}) and (\ref{eq Qgood Hausdorff}) there exist $z\in F^{q(\a)-1}(\tilde \gamma)$, $z'\in F^{q(\a)}(\tilde \gamma)$ 
satisfying  
\ary\label{eq Qgood points}
									d(z,z')\leq A_0(\cG(\a), K_{1}^{q(\a)}).
\eary	
%
By Lemma \ref{lem brouwer curve compactness} there are admissible coordinates $\phi$ for the region bounded by $(\gamma, f(\gamma))$ so that 
$\|\phi\|_{\diff^1([0,1]^2)}\leq C_1( \a^{-1}, K_{3})$.   
Since the inverse of $F^{q(\a)-1}\circ\phi$ maps $z$ and $z'$ to points at least distance $1$ apart, we must have  
$1\leq\|D\phi^{-1}\|\| D(f^{-1})^{q(\a)-1}\|d(z,z')$.   Hence 
%
$1\leq C_1(\a^{-1}, K_{3}) K_1^{q(\a)-1}A_0(\cG(\a), K_{1}^{q(\a)})$, 
which by monotonicity contradicts \eqref{E:condition G(a)}.   
\end{proof}

\section{Smooth domain bounded by good curves}\label{S:renorm}
Throughout this section $f\in\diff^{\infty}(\A)$ denotes a pseudo-rotation with rotation number $\a$.   

\subsection{Renormalization of pseudo-rotations}\label{S:renorm def}
 Given an integer $n\in\N$, we denote by  $F_n$ \textit{the unique lift of $f^n$ to $\tilde \A$ such that $\rho(F_n) \in (0, 1)$}.  
 Suppose that $H\in\diff^\infty(\tilde\A)$ has constant Jacobian and ``pulls-back'' $F_n$ to $T$, more precisely  
 \ary\label{eq H}
						F_nH=HT.
\eary
Such an $H$ arises as follows: If $\gamma_n$ is a smooth Brouwer curve for $f^n$, and $\Omega_n\subset\tilde \A$  is the closed region bounded by $(\tilde \gamma_n,F_n(\tilde \gamma_n))$, where $\tilde \gamma_n$ is an arbitrary lift of $\gamma_n$, then by Lemma \ref{A:admissible-coordinates-modified} there exists an admissible coordinate 
$H : [0,1]^2 \to  \Omega_n$ for $(\Omega_n,  F_n)$ with respect to $f^n$, and this uniquely extends to a  $C^\infty$ diffeomorphism of $\tilde \A$ with constant Jacobian satisfying \eqref{eq H}.   

 Set $\alpha := \rho(f)  \in (0,1) \setminus \Q$, then $\rho(F_1) = \a$ and $F_n=T^{-\lfloor n\a\rfloor}F_1^n$.   
 More generally we abbrieviate 
 \aryst
 			F^{a,b} := T^{b}F_1^a \qquad \forall a,b \in \Z,
\earyst
accounting for all lifts of all powers of $f$.   In this notation $F_n = F^{n, -\lfloor n\a \rfloor}$.
Denote by $J: \tilde \A \to \tilde \A$ the orientation reversing diffeomorphism
\aryst
				J(x, y) = (-x,y).
\earyst
By \eqref{eq H} the ``pull-back'' of $F_n$ by $HJ$ is $J^{-1}TJ=T^{-1}$: 
\ary \label{eq jtjtinverse}
 (HJ)^{-1}F_n(HJ) = T^{-1}.
\eary
Note that as $F^{a,b}$ and  $F_n$ commute, 
it follows that their respective ``pull-backs'' by $HJ$ commute, i.e.\ that $(HJ)^{-1}F^{a,b}(HJ)$ commutes with $T^{-1}$ (equivalently, with $T$).    
Consequently $(HJ)^{-1}F^{a,b}(HJ)$ descends to a $C^\infty$-smooth diffeomorphism of the annulus, which we call the renormalization of $f$:

\begin{defi}\label{D:renorm}
Let $f$ be a pseudo-rotation.   Given $n\in\N$ and a diffeomorphism $H:\tilde\A\to\tilde\A$ satisfying \eqref{eq H}, we define the 
{\it renormalization of $f$} (with respect to $H$ and $(n,a,b)\in\N\times\Z\times\Z$) to be the unique diffeomorphism 
\aryst
							f^{a,b}_{H}:\A\to\A
\earyst
which lifts to 
\ary\label{eq fabh}
			F^{a,b}_H := (HJ)^{-1}F^{a,b}(HJ).  
\eary
\end{defi}

From the definition, $f^{a,b}_{H}$ is $\omega$-preserving.     
It is also a pseudo-rotation: 
\begin{lemma}\label{renorm is prot} 
	Let $f$ be a pseudo-rotation with $\rho(f) = \a$ and $n\in\N$.    Then for each $a, b \in \Z$ with $a \lfloor n\a \rfloor + b n \neq 0$, the renormalization of $f$ with respect to $H$ and 
	$(a,b)$, is also a pseudo-rotation and its rotation number is given by 
	\ary\label{E:rho of renorm} 
	\rho(f^{a,b}_H) = \left\{ \frac{ -  a \a - b }{ \{ n\a \}} \right\}, 
	\eary 
	where $\{\cdot\}$ denotes the fractional part, see \eqref{E:frac}.
\end{lemma}
\begin{proof}
	 We first show that $f^{a,b}_H$ has no periodic points.
	Assume to the contrary there are integers $p \in \Z$, $q > 0$ and some $z \in \tilde\A$
    such that $(F^{a,b}_H)^q(z) = T^p(z)$.   
    Then by \eqref{eq jtjtinverse} and \eqref{eq fabh} we have
	\ary \label{eq periodicpoint}
	z  = (F^{a,b}_H)^q T^{-p}(z) = (HJ)^{-1} F_1^{(qa + pn)} T^{ (qb - p\lfloor n\a \rfloor ) }  H J(z).
	\eary
	However, \eqref{eq periodicpoint} and the condition $a \lfloor n\a \rfloor + bn \neq 0$ implies that $HJ(z)$ descends to a perodic point for $f$, which is impossible as $f$ is a 		pseudo-rotation.   Hence $f^{a,b}_H$ is a pseudo-rotation in $\diff^\infty(\A, \omega)$.
	
	To compute $\rho(f^{a,b}_H)$, it suffices to study the trajectory of an arbitrary $z \in \tilde\A$ under the iterates of $F^{a,b}_H=(HJ)^{-1}T^{b}F_1^a(HJ)$, for example a point on the boundary.    A standard argument leads to $\rho(F^{a,b}_H)=-\rho(F^{a,b})/\rho(F_n)=-(a\a+b)/\{n\a\}$, from which the expression \eqref{E:rho of renorm} follows.  
	\end{proof}

\begin{lemma}\label{L:bounds on renorm}
	Let $f$ be a pseudo-rotation with $\rho(f) = \a$, and suppose $n,k\in\N$ with $n<q_k$.    Suppose $H\in\diff^\infty(\tilde\A)$ satisfies \eqref{eq H}.   
	Then for all $a,b\in\Z$, if $q_{k+1}$ is sufficiently large compared with $q_k,\, \|H\|_{C^1([0,1]^2)}$,  
	$\|f\|_{\diff^1([0,1]^2)}$, $a,\,b$, then 
	for each $r\in\N$ the renormalization $f^{a,b}_{H}$ satisfies 
	\ary\label{E:bounds on renorm}
			\ \ \ \norm{f^{a,b}_H}_{\diff^r(\A)} \leq E_r\Big(\|f\|_{\diff^r(\A)},\,\|H\|_{\diff^r([0,1]^2)},\,q_k,\,a,\,b\Big), 
	\eary
 	for some increasing function $E_r :\R_+^{ 2}\times \N^3 \to \R_+$ that depends only on $r$.  
\end{lemma}
\begin{proof}  
Fix $a,b\in\Z$.    We claim that if $q_{k+1}$ is sufficiently large compared with $q_k,\, \|H\|_{C^1([0,1]^2)}$, $\|f\|_{\diff^1([0,1]^2)}$, $a,\,b$, then 
	\ary\label{D:L}
				F^{a,b}_{H}([0,1]^2) \subset [-N,N+1]\times[0,1],     
	\eary
	for some $N\in\N$ bounded by a function of $q_k,\, \|H\|_{C^1([0,1]^2)}$, $\|f\|_{C^1(\A)}$, $a,\,b$. 
To prove this claim we abbrieviate $\tilde\gamma=\tilde\gamma_n$ and find $N\in\N$ satisfying
\ary\label{D:L 2}
				(F^{a,b}_{H})^N(\tilde\gamma)> F^{a,b}(\tilde\gamma), 
\eary
since a short computation shows that this implies \eqref{D:L}.    

Since $\tilde\gamma$ is parameterised by 
$H$ on $\{0\}\times[0,1]$, the length satisfies 
$\ell(\tilde\gamma)\leq L_0$ 
where $L_0\in\N$ is bounded by $\|H\|_{C^1([0,1]^2)}$.    
 Note that there exists $M\in\N$ bounded in terms of $L_0,\,\|f\|_{C^1(\A)}$, $a,\,b$, for which 
\ary\label{D: M}
					T^M(\tilde\gamma) > F^{a,b}(\tilde\gamma).
\eary
Indeed, this follows easily from the fact that the length of $F^{a,b}(\tilde\gamma)=T^bF_1^a(\tilde\gamma)$ is bounded by $\|Df\|^a_{C^0(\A)}$, 
and $F^{a,b}$ moves a point on the boundary at most $\lceil\rho(F^{a,b})\rceil=\lceil a\a+b\rceil$.   This proves \eqref{D: M}.    Now set 
$N:=q_kN_0$, 
where $N_0\in\N$ is sufficiently large that 
\ary\label{E:lower bound N}
				\rho(F_n)^N > M + L_0 +1.   
\eary
Since $\rho(F_n^N)=N\{n\a\}$, we can assume that $N$ is bounded by a function of $M, L_0$ and $1/\{n\a\}\leq 1/\{q_{k-1}\a\}=1/\beta_{k-1}\leq 2q_k$.   By Theorem \ref{thm aflxz}
\ary\label{E:C0 bound on f^nN}
		d_{C^0(\A)}(f^{nN},\id)< A_0\big(\norm{\rho(f^{nN})}_{\R / \Z},\,  \norm{D(f^{nN})}\big).  
\eary
Since $\norm{\rho(f^{nN})}_{\R / \Z}\leq nN_0\norm{q_k\a}_{\R / \Z}=nN_0\beta_k\leq nN_0/q_{k+1}\leq N/q_{k+1}$, while  $\norm{D(f^{nN})}\leq \norm{f}^{q_kN}_{C^1(\A)}$ 
it follows that the right hand side of \eqref{E:C0 bound on f^nN} is less than $1/2$, if $q_{k+1}$ is sufficiently large compared with $q_k, \norm{f}_{C^1(\A)}$ and $N$, i.e.\ 
compared with $q_k, \norm{f}_{C^1(\A)}, M, L_0$.   In this case, by Lemma \ref{lem:lift} 
\ary\label{E:N and P}
		\|T^{-P}(F_n)^{N}-\id_{\tilde\A}\|_{C^0(\tilde\A)}<1/2, 
\eary
where $P:=\lfloor\rho(F_n)^N\rfloor$, because $T^{-P}(F_n)^{N}$ is the unique lift of $f^{nN}$ with rotation number in $(-1/2,1/2)$.   Indeed,  by choice of $P$, 
$T^{-P}(F_n)^{N}$ has rotation number in $(0,1)$, but moreover it lies in $(0,1/2)$ because the fractional part is 
$\{nN\a\}=\{nN_0q_k\a\}\leq nN_0\{q_k\a\}\leq N/q_{k+1}$ which without loss of generality is also less than $1/2$.    
From \eqref{E:N and P} 
$d_{H}\big((F_n)^{N}(\tilde\gamma),\, T^{P}(\tilde\gamma)\big) < 1/2$ so 
\aryst
							(F_n)^{N}(\tilde\gamma) > T^{P-L_0-1}(\tilde\gamma)> T^{M}(\tilde\gamma)
\earyst 
using the definition of $P$ and \eqref{E:lower bound N} for the second inequality.   Combining this with \eqref{D: M} gives \eqref{D:L 2} under the mentioned largeness assumptions on  
 $q_{k+1}$.    Since $M$ is bounded in terms of $L_0,\|f\|_{C^1(\A)}$, $a,\,b$, while $L_0$ is bounded by $\|H\|_{C^1([0,1]^2)}$, 
this proves \eqref{D:L 2} and thus \eqref{D:L} as claimed.   

Now we prove \eqref{E:bounds on renorm}.    
From the definitions $\|f^{a,b}_H\|_{\diff^r(\A)}=\|H^{-1}F^{a,b}H\|_{\diff^r([0,1]^2)}$ is bounded in terms of $\|H\|_{\diff^r([0,1]^2)}$, $\|F^{a,b}\|_{\diff^r(H([0,1]^2))}\leq\|f\|^a_{\diff^r(\A)}$, 
and 
\aryst
	 \|H^{-1}\|_{\diff^r(F^{a,b}H([0,1]^2))}=\|H\|_{\diff^r(H^{-1}F^{a,b}H([0,1]^2))}\leq \|H\|_{\diff^r([-N,N+1]\times[0,1])}
\earyst 
using \eqref{D:L}.   So it remains to bound $\|H\|_{\diff^r([-N,N+1]\times[0,1])}$.     
For each $j\in\Z$, $\|H\|_{\diff^r([j,j+1]\times[0,1])}=\|HT^j\|_{\diff^r([0,1]^2)}=\|F_n^jH\|_{\diff^r([0,1]^2)}$ 
is bounded by $\|H\|_{\diff^r([0,1]^2)}$ and 
$\|F_n^j\|_{\diff^r(\tilde\A)}=\|f^{nj}\|_{\diff^r(\A)}$.     Therefore $\|H\|_{\diff^r([-N,N+1]\times[0,1])}$ is bounded in terms of $\|f\|^{nN}_{\diff^r(\A)}$ and $\|H\|_{\diff^r([0,1]^2)}$.   
Putting this all together, and inserting the bounds on $N$, we arrive at \eqref{E:bounds on renorm}.  
\end{proof}

\subsection{Finding a good Brouwer curve}\label{SS:finding brouwer curve}
We will consider a renormalization of the form $f^{a,b}_{H}$ where $(a,b)$ is in the best approximation sequence $(q_j,p_j)_{j\in\N_0}$.    

Consider $n\in\N$ odd, so that $q_{n-1} \a - p_{n-1} >0$.   Suppose that 
$\sigma\subset\A$ is a smooth Brouwer curve for $f^{q_{n-1}}$, that $\tilde\sigma\subset\tilde\A$ is an arbitrary lift, and that $H:[0,1]^2\to\Omega$ is an admissible coordinate for $F_{q_{n-1}}:=F^{q_{n-1},-p_{n-1}}$, where $\Omega\subset\tilde\A$ is the closed region bounded by $(\tilde\sigma,F_{q_{n-1}}(\tilde\sigma))$.     We also denote by $H\in\diff^\infty(\tilde\A)$ the unique smooth extension satisfying  
\ary\label{eq H again}
						F_{q_{n-1}}H=TH,   
\eary
which of course also has constant Jacobian everywhere.   We denote by 
\ary\label{D:renorm again}
		f^{q_{n},-p_{n}}_{H}:\A\to\A
\eary
the renormalization of $f$ with respect to $H$ and $(q_{n-1},q_{n},-p_{n})$.   That is, $f^{q_{n},-p_{n}}_{H}$ is the descendent to the annulus of 
\ary\label{D: Stilde}
						\tilde S:=(HJ)^{-1}F^{q_{n},-p_{n}}(HJ):\tilde\A\to\tilde\A. 
\eary 
 (So $q_{n-1}$ is playing the role that $n$ did in Section \ref{S:renorm def}, while $(q_{n},-p_{n})$ are playing the 
role of $(a,b)$.)   
By Lemma \ref{renorm is prot} $f^{q_{n},-p_{n}}_{H}$ is a pseudo-rotation with 
\ary\label{D: renorm again}
							\rho\big(f^{q_{n},-p_{n}}_{H}\big)=\a_{n} \in (0, 1/2).   
\eary 
The next result shows that a Brouwer curve for $f^{q_{n},-p_{n}}_{H}$ gives rise to a Brouwer curve for $f$ that is also a Brouwer curve for many higher iterates of $f$.     
 
\begin{prop}\label{T:brouwer for renorm}
Fix $r\in\N$ and $n\in\N_{\mathrm{odd}}$.  Suppose that the renormalization $f^{q_{n},-p_{n}}_{H}$ described above, has an $a_{n+1}$-good Brouwer curve $\gamma'$.   Then: 
\begin{enumerate}
\item $\gamma'$ pushes forward via $HJ$, in the following sense, to a $q_{n+1}$-good Brouwer curve $\gamma$ for $f$.  
Namely, take any lift $\tilde\gamma'\in\tilde\A$ of $\gamma'$, then 
set $\tilde\gamma:=HJ\tilde\gamma'\in\tilde\A$ and 
set $\gamma:=\pi\tilde\gamma$.  In particular $\gamma$ is a Brouwer curve for $f^{q_{n}}$.  
\item If $\Phi'$ is an admissible coordinate for the closed region bounded by  
$\big(\gamma',f^{q_{n},-p_{n}}_{H}(\gamma')\big)$ with respect to $f^{q_{n},-p_{n}}_{H}$, then the region $\big(f^{q_{n}}(\gamma),\,\gamma\big)$ 
has admissible coordinates $\Phi$, with respect to the map $f^{-q_{n}}$, satisfying 
\begin{equation}\label{E:bound Phi}
	\|\Phi\|_{\diff^r([0,1]^2)}\leq G_{r}\Big( \|\Phi'\|_{\diff^r([0,1]^2)},\,\|H\|_{\diff^r([0,1]^2)},\,\|f\|_{\diff^r(\A)},\, q_{n-1}\Big)
\end{equation}
for some increasing function $G_{r}:\R^3_+\times\N\to\R^+$ depending only on $r$.   
\end{enumerate}
\end{prop}
\begin{proof}	
	We abbrieviate 
	\aryst
		S_0 := F^{q_{n-1}, -p_{n-1}},\qquad S := F^{q_{n}, -p_{n}}.  
	\earyst
	Thus \eqref{eq H again} and \eqref{eq jtjtinverse} become respectively  
	\aryst
	H T = S_0 H,\ \qquad (HJ)^{-1}S_0(HJ) = T^{-1}, 
	\earyst
	while \eqref{D: Stilde} becomes $\tilde S := (HJ)^{-1}S(HJ)$.   
	Note that $\tilde S$ commutes with $T$, and descends to $f^{q_{n}, -p_{n}}_H$. Moreover $\rho(\tilde S) = \rho(f^{q_{n}, -p_{n}}_H) = \a_{n} \in (0, \frac{1}{2})$.

\smallskip

{\bf Proof of Item (2):}     By assumption there exists an admissible coordinate $\Phi':[0,1]^2\to\cR'$, with respect to the renormalized 
map $f^{q_{n},-p_{n}}_{H}$, where $\cR'\subset\A$ is the closed region bounded by $\big(\gamma',f^{q_{n},-p_{n}}_{H}(\gamma')\big)$.   
This lifts via $\pi$ to an admissible coordinate $\tilde{\Phi}'$ for the region 
$\tilde{\cR}'\subset\tilde\A$ bounded by $\big(\tilde{\gamma}',\tilde S(\tilde{\gamma}')\big)$, 
with respect to $\tilde S$.   Without loss of generality we may choose $\tilde\Phi'$ 
so that 
\ary\label{D:N}
						\tilde\Phi'([0,1]^2)\subset [-N,N+1]\times[0,1]
\eary
where $N\leq \|\Phi'\|_{C^1([0,1])}$, for example if we arrange that $\tilde{\Phi}'(0,0)\in[0,1]^2$.   
In any case, $HJ\circ\tilde{\Phi}':[0,1]^2\to\tilde\A$ has image 
$\tilde{\cR}:=HJ(\tilde{\cR}')$ bounded by $\big(S(\tilde{\gamma}),\tilde{\gamma}\big)$, and defines an admissible coordinate for $S^{-1}$.   
Hence, 
\ary\label{coords from renorm}
			\Phi:=\pi\circ HJ\circ\tilde{\Phi}':[0,1]^2\to \cR
\eary
is a diffeomorphism onto $\cR:=\pi(\tilde\cR)\subset\A$ bounded by $\big(f^{q_{n-1}}(\gamma),\gamma\big)$, where $\gamma:=\pi\tilde\gamma$, and defines an admissible 
coordinate for $f^{-q_{n-1}}$.    For this last step we must check that $\pi$ restricts to a diffeomorphism from $\tilde{\cR}$ to $\cR$.   For this it suffices to show that 
\ary\label{E:Sgamma}
					T^{-1}(\tilde\gamma)<S(\tilde\gamma)<\tilde\gamma, 
\eary
which we now prove.   Since $\gamma'\subset\A$ is a Brouwer curve for $f^{q_{n},-p_{n}}_{H}$, 
the lift $\tilde\gamma'$ satisfies $\tilde S(\tilde\gamma')\cap T^{k}(\tilde\gamma')=\emptyset$ for all $k\in\Z$.   Applying $HJ$ with $k=0$ gives 
\ary\label{E:condition 1 gamma}
					S(\tilde\gamma)\cap \tilde\gamma=\emptyset.  
\eary
Since $\gamma'\subset\A$ has no self-intersections, the lift $\tilde\gamma'$ also satisfies $\tilde\gamma'\cap T^{j}(\tilde\gamma')=\emptyset$ for all $j\in\Z\backslash\{0\}$, so applying $HJ$ with $j=-1$ gives 
\ary\label{E:condition 2 gamma}
					\tilde\gamma\cap S_0(\tilde\gamma)=\emptyset.  
\eary
Combining \eqref{E:condition 1 gamma} and \eqref{E:condition 2 gamma} with $\rho(S_0)=q_{n-1}\a-p_{n-1}>0$ and $\rho(S)=q_{n}\a-p_{n}<0$ gives 
    \ary \label{eq orderofss0gamma'}
    						S(\tilde\gamma) < \tilde\gamma < S_0(\tilde\gamma).
    \eary
This proves the right inequality in \eqref{E:Sgamma}.    By definition and by \eqref{relation pn qn}, we have $S^{q_{n}}_0S^{-q_{n-1}}=T$.  Combining with \eqref{eq orderofss0gamma'}, leads to $T^{-1}(\gamma') < S(\gamma')$, completing the proof of \eqref{E:Sgamma}.  

Hence $\Phi$ in \eqref{coords from renorm} is an admissible coordinate and it remains to show \eqref{E:bound Phi}.   
From \eqref{coords from renorm} and \eqref{D:N}, $\|\Phi\|_{\diff^r([0,1]^2)}$ is controlled by $\|\Phi'\|_{\diff^r([0,1]^2)}$ and 
$\|H\|_{\diff^r(\tilde\Phi'([0,1]^2)}\leq \|H\|_{\diff^r([-N,N+1]\times[0,1])}$.   
Arguing as at the end of Lemma \ref{L:bounds on renorm}, the latter is bounded in terms of $\|H\|_{\diff^r([0,1]^2)}$ 
and $\|f^{q_{n-1}}\|^N_{\diff^r(\A)}$.    Since $N\leq \|\Phi'\|_{\diff^1([0,1]^2)}$, we arrive at \eqref{E:bound Phi}, proving Item (2). 

\smallskip

{\bf Proof of Item (1):} 
It remains to prove that $ \gamma$ is a $q_{n+1}$-good Brouwer curve for $f$.   Equivalently, that 
	$F^{k,p}(\tilde\gamma)\cap\tilde\gamma=\emptyset$ for all $k\in\{1,\ldots,q_{n+1}-1\}$ and all $p\in\Z$, 
	where in the notation from Section 4.1, $F^{k,p}=F_1^kT^p$, where $F_1$ is the unique lift of $f$ with rotation number in $(0,1)$.   
	Arguing indirectly, we find $p\in\Z$ and $k\in\{1,\ldots,q_{n+1}-1\}$ so that 
	\ary\label{intersect 1}
				\qquad	F^{k,p}(\tilde\gamma)\cap\tilde\gamma\neq\emptyset.  
	\eary   
	To get a contradiction we consider two cases separately: 
	\smallskip
	
	{\bf Case I:} \ 
	In which $\rho(F^{k,p})< 0$, i.e.\ $k\a + p < 0$.     
	\medskip
	
	Without loss of generality we can assume that $k\alpha+p$ is maximal with these properties.   In other words, 
	that $F^{k', p'}( \tilde\gamma) \cap  \tilde\gamma = \emptyset$ for all $p'\in\Z$ and all $k'\in\{1,\ldots,q_{n+1}-1\}$ satisfying $k\alpha+p<k'\alpha+p'<0$. 
	\smallskip

	$\bullet$
	By the maximality condition we can reduce to the following two subcases:   Either 
	\ary\label{case 1a}
				  k \geq q_{n+1} - q_{n-1}
	\eary
	or 
	\ary\label{case 1b}
				p_{n-1}-q_{n-1} \a  < k \a + p. 
	\eary
	Indeed, otherwise $k < q_{n+1} - q_{n-1}$ and $p_{n-1}-q_{n-1} \a  \geq  k \a + p$.   So setting $k'=k+q_{n-1}$ and $p'=p-p_{n-1}$ we have $k' \in\{1,\ldots,q_{n+1}-1\}$, and 
	\aryst
					k'\a + p'=k\a+p+(q_{n-1}\a-p_{n-1}).
	\earyst
	Hence as $0<q_{n}\a-p_{n}\leq-(k \a + p)$, because \eqref{case 1b} does not hold, we see 
	\aryst
	k \a + p < k'\a + p' < 0, 
	\earyst	
	where we have used that $k'\a+p'$ is irrational to obtain a strict inequality on the right.   
	Thus by the maximality condition on $(k,p)$ we have $F^{k', p'}( \tilde\gamma) \cap  \tilde\gamma = \emptyset$.   From 
	$\rho(F^{k',p'})=k'\alpha+p'<0$,  we see that the end points of $F^{k', p'}( \tilde\gamma)$ are on the left side of those of $\tilde\gamma$, and so we conclude that $\tilde\gamma>F^{k', p'}( \tilde\gamma)$.    
	Therefore $\tilde\gamma> F^{k', p'}( \tilde\gamma)=F^{k + q_{n-1}, -p_{n-1} + p}( \tilde\gamma)=F^{k, p}S_0( \tilde\gamma)>F^{k, p}( \tilde\gamma)$ where the last inequality is from the 
	second inequality in \eqref{eq orderofss0gamma'}.    This however contradicts \eqref{intersect 1}.    
	\smallskip
	
	$\bullet$ Now we consider the case that \eqref{case 1a} holds, which by \eqref{q iteration} implies that 
	\ary\label{case 1ab}
			q_{n+1}>k\geq q_{n}.  
	\eary   
	First suppose $k=q_{n}$.    We claim that $p>-p_{n}$.   Indeed, if not, $F^{k,p}=F^{q_{n},p}=T^{p+p_{n}}F^{q_{n},-p_{n}}=T^{p+p_{n}}S$ with $p+p_{n}\leq0$, 
	so from \eqref{eq orderofss0gamma'} $T^{p+p_{n}}S(\tilde\gamma)<\tilde\gamma$.  Therefore $F^{k,p}(\tilde\gamma)<\tilde\gamma$ contradicting \eqref{intersect 1}.   
	Thus $p>-p_{n}$, and so
	\aryst
						0>k\a + p > q_{n}\a - p_{n} = -\{q_{n}\a\}.
	\earyst 
	Hence $d(k\a,\Z)< d(q_{n}\a,\Z)$, so $k\geq q_{n+1}$ contradicting $k=q_{n}$.   
	
	It remains to consider \eqref{case 1ab} with strict inequalities: 
	\ary\label{case 1abb}
			q_{n+1}>k> q_{n}.  
	\eary   
	Set $k':=k-q_{n}$ and $p':=p+p_{n}$.   
	By \eqref{case 1abb} we have $k' > 0$, so that $k'\in\{1,\ldots,q_{n+1}-1\}$, and 
	\ary\label{case 1d}
					k'\a + p'=k\a+p-(q_{n}\a-p_{n}). 
	\eary
	We claim that 
	\ary\label{case 1c}
					k\a+p < k'\a + p'<0.   
	\eary
	Indeed, 
	 \eqref{case 1d} with $q_{n}\a-p_{n}=-\{q_{n}\a\}<0$ implies the first inequality.   
	To see the second inequality in \eqref{case 1c}, note that $q_{n}< k<q_{n+1}$ implies $d(k\a,\Z)> d(q_{n}\a,\Z)$.  
	Therefore $|k\a+p|> d(q_{n}\a,\Z)=\{q_{n}\a\}=p_{n}-q_{n}\a$ and since $k\a+p<0$ this means $k\a+p<q_{n}\a -p_{n}$.   Therefore from 
	\eqref{case 1d} we get $k'\a + p'< 0$ as claimed.  This proves \eqref{case 1c}.    
	
	Hence, from the maximality condition on $(k,p)$ we conclude $F^{k', p'}( \tilde\gamma) \cap  \tilde\gamma = \emptyset$, and therefore moreover $\tilde\gamma>F^{k', p'}( \tilde\gamma)$, 
	because $\rho(F^{k',p'})=k'\alpha+p'<0$.    Thus $\tilde\gamma> F^{k', p'}( \tilde\gamma)=F^{k, p}S^{-1}( \tilde\gamma)>F^{k, p}( \tilde\gamma)$ where the last inequality is from the 
	first inequality in \eqref{eq orderofss0gamma'}.    This however contradicts \eqref{intersect 1}.   	
	
	\smallskip
	
	$\bullet$ Now we consider the case that \eqref{case 1b} holds, namely  	
	\aryst
				-\{q_{n-1}\a\}=p_{n-1}-q_{n-1} \a  < k \a + p <0.  
	\earyst
	Therefore $d(k\a,\Z)<\{q_{n-1}\a\}=d(q_{n-1}\a,\Z)$.    It is a basic fact about rational approximations of irrational $\a\in(0,1/2)$, that the only integer multiples $k\a$ 
	that are closer to $\Z$ than $d(q_{n-1}\a,\Z)$, with $k<q_{n+1}$ and $n$ odd,  
	for which $k\a$ approaches $\Z$ from the ``negative side'', are the multiples of $q_{n}$.    Therefore 
	\aryst
	\qquad k = iq_{n},\qquad\mbox{some }i\in\{1,\ldots,a_{n+1}\}.  
	\earyst
	We claim that $p = -ip_{n}$.    Indeed, from 
	\eqref{eq continuefraction} we 
	have $-1<-i\beta_{n}<0$, that is $-1< i(q_{n} \a - p_{n})<0$, i.e.\ 
	\aryst
	-1 < k\a -ip_{n} <  0. 
	\earyst
	As also $-1 < k\a + p <  0$, we must have $p = -ip_{n}$ as claimed.     
	
	Consider first the case $i=1$, that is, $k = q_{n}$.   Then $F^{k,p}=F^{q_{n},-p_{n}}=S$ 
	so that by \eqref{E:condition 1 gamma} we have $F^{k,p}(\tilde\gamma)\cap\tilde\gamma=\emptyset$ contradicting \eqref{intersect 1}.  
	
	It remains to consider the case $i>1$.    Then $-1<-(i-1)\beta_{n}<0$ leads to 
	$-1 < (k-q_{n}) \a -(i-1)p_{n} <  0$, i.e.\ 
	\ary\label{case 1e}
	-1 < (k-q_{n}) \a +p + p_{n} <  0.
	\eary
	Note that $\rho(S^{-1}F^{k, p})=\rho(F^{k-q_{n},p+p_{n}})=(k-q_{n}) \a +p + p_{n}$, so that by \eqref{case 1e} 
	we have $\rho(S^{-1}F^{k, p})<0$.   Therefore the endpoints of $S^{-1}F^{k, p}( \tilde\gamma)$ are on the left hand side of those of $\tilde\gamma$.   
	On the other hand, by \eqref{eq orderofss0gamma'} we have 
	$S^{-1} F^{k, p}( \tilde\gamma)= F^{k, p}S^{-1}( \tilde\gamma) > F^{k,p}( \tilde\gamma)$.   This is a contradiction.    
	
	\smallskip
	
	{\bf Case II:} \
	In which $\rho(F^{k,p})> 0$, i.e.\ $k\a + p > 0$.     
	\smallskip
	
	Without loss of generality we can assume that $k\alpha+p$ is minimal with the given restrictions.   In other words, 
	that $F^{k', p'}( \tilde\gamma) \cap  \tilde\gamma = \emptyset$ for all $p'\in\Z$ and all $k'\in\{1,\ldots,q_{n+1}-1\}$ satisfying $k\alpha+p>k'\alpha+p'>0$. 
	\smallskip
	
	$\bullet$
	By the minimality condition we can reduce to the following two subcases:   Either 
	\ary\label{case 2a}
				  k \leq q_{n-1}
	\eary
	or 
	\ary\label{case 2b}
				q_{n-1} \a -p_{n-1} > k \a + p.  
	\eary
	Indeed, otherwise $q_{n+1}>k > q_{n-1}$ and $q_{n-1} \a -p_{n-1}  \leq  k \a + p$.  So setting $k'=k-q_{n-1}$ and $p'=p+p_{n-1}$ we have $k' \in\{1,\ldots,q_{n+1}-1\}$, and 
	\aryst
					k'\a + p'=k\a+p-(q_{n-1}\a-p_{n-1}).
	\earyst
	Hence as $0<q_{n}\a-p_{n}\leq k \a + p$, because \eqref{case 2b} does not hold, we see that 
	\aryst
	k \a + p > k'\a + p' > 0
	\earyst	
	where we have used that $k'\a+p'$ is irrational to obtain a strict inequality on the right.    
	Thus by the minimality condition on $(k,p)$ we have $F^{k', p'}( \tilde\gamma) \cap  \tilde\gamma = \emptyset$.   From 
	$\rho(F^{k',p'})=k'\alpha+p'>0$,  we see that the end points of $F^{k', p'}( \tilde\gamma)$ are on the right side of those of $\tilde\gamma$, and we conclude that $\tilde\gamma<F^{k', p'}( \tilde\gamma)$.    
	Therefore $\tilde\gamma< F^{k', p'}( \tilde\gamma)=F^{k - q_{n-1}, p+p_{n-1}}( \tilde\gamma)=F^{k, p}S^{-1}_0( \tilde\gamma)<F^{k, p}( \tilde\gamma)$ using the 
	second inequality in \eqref{eq orderofss0gamma'}.    This contradicts \eqref{intersect 1}.   
	\smallskip

	$\bullet$ Here we consider the case that \eqref{case 2a} holds, that is,  
	\ary\label{case 2ab}
			0<k \leq q_{n-1}.
	\eary
	First suppose $k=q_{n-1}$.    We claim $p<-p_{n-1}$.   If not, $F^{k,p}=F^{q_{n-1},p}=T^{p+p_{n-1}}F^{q_{n-1},-p_{n-1}}=T^{p+p_{n-1}}S_0$ with $p+p_{n-1}\geq0$, 
	so from \eqref{eq orderofss0gamma'} $T^{p+p_{n-1}}S_0(\tilde\gamma)>\tilde\gamma$.  So $F^{k,p}(\tilde\gamma)>\tilde\gamma$ contradicting \eqref{intersect 1}.   
	This proves $p<-p_{n-1}$, and so
	\aryst
						0<k\a + p < q_{n-1}\a - p_{n-1} = \{q_{n-1}\a\}.
	\earyst 
	But this means that $d(k\a,\Z)< d(q_{n-1}\a,\Z)$, which implies $k\geq q_{n}$ contradicting $k=q_{n-1}$.   
	
	It therefore remains to consider \eqref{case 2ab} with strict inequalities: 
	\ary\label{case 2abb}
			0<k < q_{n-1}.
	\eary   
	Setting $k':=k+q_{n}$ and $p':=p-p_{n}$,  by \eqref{q iteration} $k' < q_{n+1}$, 
	 so $k'\in\{1,\ldots,q_{n+1}-1\}$, and 
	\ary\label{case 2d}
					k'\a + p'=k\a+p+(q_{n}\a-p_{n}).   
	\eary
	We claim that 
	\ary\label{case 2c}
					k\a+p > k'\a + p'>0.   
	\eary
	Indeed, \eqref{case 2d} with $q_{n}\a-p_{n}=-\{q_{n}\a\}<0$
	implies the first inequality.   
	To see the second inequality in \eqref{case 2c}, note that \eqref{case 2abb} implies $d(k\a,\Z)>d(q_{n-1}\a,\Z)$.    
	Therefore $k\a+p=|k\a+p|> d(q_{n-1}\a,\Z)=\{q_{n-1}\a\}>\{q_{n}\a\}=p_{n}-q_{n}\a$, which means that the right hand side of \eqref{case 2d} is positive, that is $k'\a + p'>0$ as claimed.  
	This proves \eqref{case 2c}.    
	
	\smallskip

	$\bullet$ Now we consider the case that \eqref{case 2b} holds, namely  	
	\aryst
			 \{q_{n-1}\a\}=q_{n-1} \a -p_{n-1} > k \a + p >0.  
	\earyst
	Therefore $d(k\a,\Z)<\{q_{n-1}\a\}=d(q_{n-1}\a,\Z)$.       	
         It follows from the theory of rational approximations, since $\a\in(0,1/2)\backslash\Q$, $n$ is odd, $k<q_{n+1}$, 
	$d(k\a,\Z)<d(q_{n-1}\a,\Z)$ and $k\a$ is close to $\Z$ from the ``positive side'', that 
	\aryst
	\qquad k = q_{n-1}+iq_{n},\qquad\mbox{some }i\in\{1,\ldots,a_{n+1}-1\}.  
	\earyst
	We claim that $p = -p_{n-1}-ip_{n}$.   Indeed, a computation using \eqref{eq continuefraction} shows that 
	$(a_{n+1}-1)\beta_{n}<\beta_{n}$.   It follows that if $i\in\{1,\ldots,a_{n+1}-1\}$ then $1>\beta_n-i\beta_{n}>0$, which becomes 
	\ary\label{case 2e}
				1>k\a - (p_{n-1}+ip_{n})>0  
	\eary
	where $k = q_{n-1}+iq_{n}$.    Since $p\in\Z$ must uniquely satisfy $1>k\a+p>0$, we conclude from \eqref{case 2e} that $p=-p_{n-1}-ip_{n}$ as claimed. 
	Thus 
	\aryst
	(k,p) \in \{(q_{n-1} + iq_{n}, -p_{n-1} - i p_{n} ) \mid 0 <  i < a_{n+1}\}.
	\earyst
	By the hypothesis that $\gamma'$ is an $a_{n+1}$-good curve,
	we have 
	\aryst 
	T^{-1}\tilde S^i(\tilde \gamma') \cap \tilde \gamma' = \emptyset \quad  \forall 1 \leq i < a_{n+1}.
	\earyst
	By $T^{-1} = JH^{-1} S_0 HJ$, $\tilde S = JH^{-1}S HJ$ and $\tilde\gamma = HJ(\tilde \gamma')$, we obtain
	\aryst
	F^{q_{n-1} + i q_{n}, -p_{n-1} - ip_{n}}(\tilde\gamma) \cap \tilde\gamma = 	S_0S^{i}(\tilde\gamma) \cap \tilde\gamma = \emptyset
	\earyst
	for each integer $0 < i < a_{n+1}$.    Hence $F^{k,p}(\tilde\gamma) \cap \tilde\gamma = \emptyset$, contradicting \eqref{intersect 1}.   
	
	\medskip 
	
	This covers all cases, and we conclude that $ \gamma$ is $q_{n+1}$-good.  
\end{proof}

\begin{thm}\label{cor smooth return domain new version}
	For each $(r,M)\in\N_{\geq 2}\times (1,\infty)$, there exist increasing functions $P=P_{(r,M)}:\N\to\R_+$ and $W_{(r,M)}:\N \to \R_+$, 
	such that the following holds.    If $f:\A\to\A$ is any smooth pseudo-rotation with $\|f\|_{\diff^{r+5}(\A)} \leq M$, 
	whose rotation number $\a=\rho(f)$ has associated sequence $q_1, q_2, \ldots $ satisfying 
	\ary\label{eq qnqn+1qn+2}
		q_{n} > P(q_{n-1}),\qquad q_{n+1} > P(q_{n}),\qquad q_{n+2} > P(q_{n+1}) 
	\eary
	for some $n\in\N_{\mathrm{odd}}$, then there exists a smooth Brouwer curve $\gamma\subset\A$ for $f$, so that: 
	\begin{enumerate}
		\item $\gamma$ is $q_{n+1}$-good for $f$.  
		\item the fundamental domain bounded by $(f^{q_n}(\gamma), \gamma)$ has an admissible coordinate $\Phi:[0,1]^2\to \A$, 
		with respect to $f^{-q_n}$, satisfying 
		\ary\label{E:bound W}
					\|\Phi\|_{\diff^r([0,1]^2)}\ \leq\ W_{(r,M)}(q_{n+1}).   
		\eary
	\end{enumerate} 
\end{thm}

%
The proof is based on the following three steps: 
\begin{enumerate} 
 \item Apply Lemma \eqref{lem brouwer curve compactness} to $f^{q_{n-1}}$: This yields a Brouwer curve for $f^{q_{n-1}}$ and admissible coordinates $H$ on a corresponding fundamental 
 domain, with $\diff^r$-bounds on $H$.   
 To do this we will need smallness assumptions on $\rho(f^{q_{n-1}})=\beta_{n-1}$ and $\G(\beta_{n-1})$.   
 From this step we will be able to define the renormalization $f^{q_{n},-p_{n}}_{H}$.  
 \item Apply Proposition \eqref{prop. main} to $f^{q_{n},-p_{n}}_{H}$:  This yields a $q(\rho(f^{q_{n},-p_{n}}_{H}))=q(\a_{n})=a_{n+1}$-good Brouwer curve for 
 $f^{q_{n},-p_{n}}_{H}$, under smallness assumptions on $\rho(f^{q_{n},-p_{n}}_{H})=\a_{n}$ and $\G(\a_{n})$.  
 \item Apply Proposition \eqref{T:brouwer for renorm} to $f^{q_{n},-p_{n}}_{H}$ and the Brouwer curve found in Step (2):  This yields a $q_{n+1}$-good Brouwer curve for $f$, with estimates on the corresponding 
 fundamental domain for $f^{-q_{n}}$.  
\end{enumerate}

\begin{proof}
We go through the above steps and a final fourth step to reformulate the conditions on $\a$ in the form of \eqref{eq qnqn+1qn+2}.   
Fix $r\in\N$ and $n\in\N_{\mathrm{odd}}$.  

\medskip

{\bf Step 1:}  
Since $n$ is odd  $\rho(f^{q_{n-1}})=q_{n-1}\a-p_{n-1}=\beta_{n-1}$, 
so to apply Lemma \eqref{lem brouwer curve compactness} to $f^{q_{n-1}}$, it suffices that 
\begin{equation}\label{E:estimates fn+1 I}
				\beta_{n-1} \leq \varepsilon_0\big(\|f^{q_{n-1}}\|^{-1}_{\diff^1(\A)}\big),\quad \G(\beta_{n-1})\leq\varepsilon_1\big(\beta_{n-1},\, \|f^{q_{n-1}}\|^{-1}_{\diff^1(\A)}\big).
\end{equation}
Assume that \eqref{E:estimates fn+1 I} holds.   Then we have a Brouwer curve $\sigma\subset\A$ for $f^{q_{n-1}}$ and 
an admissible coordinate $H:[0,1]^2\to\Omega$ for $F_{q_{n-1}}:=F^{q_{n-1},-p_{n-1}}$, where $\Omega\subset\tilde\A$ is a corresponding fundamental 
domain for $F_{q_{n-1}}$, with bounds 
\ary\label{T2:H estimates 2}
 			\|H\|_{\diff^{r+1}([0,1]^2)}\ \leq\ C_{r+1}\left(2q_n,\, \|f^{q_{n-1}}\|_{\diff^{r+3}(\A)}\right) 
\eary
from Lemma \ref{lem brouwer curve compactness}, using that $1/\beta_{n-1}\leq 2q_n$ from \eqref{eq continuefraction}.    Now we consider the renormalization $f^{q_{n},-p_{n}}_{H}$ with respect to $H$, 
as at the beginning of \ref{SS:finding brouwer curve}.       

By Lemma \ref{L:bounds on renorm} we have bounds on $\|f^{q_{n},-p_{n}}_{H}\|_{\diff^{r+1}(\A)}$, provided 
$q_{n+1}$ is sufficiently large compared with $\|H\|_{C^1([0,1]^2)}$, $\|f\|_{C^1(\A)}$, $q_{n-1}, q_{n}, p_{n}$.    In particular, using \eqref{T2:H estimates 2}, 
compared with $M, q_{n-1}, q_{n}$ (we ignore $p_n$ as this is bounded by a function of $q_n$).   
Without loss of generality we may assume this condition on $q_{n+1}$ due to the middle inequality in \eqref{eq qnqn+1qn+2}.    
This yields 
\ary\label{E:estimates renorm I}
							\|f^{q_{n},-p_{n}}_{H}\|_{\diff^{r+2}(\A)} \leq  E'_{r}\big(\|f\|_{\diff^{r+4}(\A)},\,q_{n}\big), 
\eary
using \eqref{T2:H estimates 2} again, for some increasing function $E'_r$ depending only on $r$.   

\smallskip

{\bf Step 2:}  
Since $\rho(f^{q_{n},-p_{n}}_{H})=\a_{n}$, to apply Proposition \ref{prop. main} it suffices that 
\begin{equation}\label{E:estimates fn+1 II}
				\a_{n}\leq\varepsilon_0\big(1/\|f^{q_{n},-p_{n}}_{H}\|_{\diff^1}\big),\quad \G(\a_{n})\leq \varepsilon_2\big(\a_{n},\, 1/\|f^{q_{n},-p_{n}}_{H}\|_{\diff^{3}}\big).
\end{equation}
Assume for the moment that these two conditions are also fulfilled.    Then we conclude that $f^{q_{n},-p_{n}}_{H}$ has a 
$q(\rho(f^{q_{n},-p_{n}}_{H})=q(\a_{n})=a_{n+1}$-good Brouwer curve $\gamma'\subset\A$ with an associated fundamental domain 
having admissible coordinates $\Phi'$ satisfying, by \eqref{L:phi estimates}, estimates of the form 
  \ary\label{T2:H estimates 3}
 			\|\Phi'\|_{\diff^{r}([0,1]^2)}\ \leq\ C_r\left(1/\a_{n},\, \|f^{q_{n},-p_{n}}_{H}\|_{\diff^{r+2}}\right).
 \eary
 

{\bf Step 3:}  
 Applying Proposition \eqref{T:brouwer for renorm} to $f^{q_{n},-p_{n}}_{H}$ with its $a_{n+1}$-good Brouwer curve $\gamma'$ found in step 2, we obtain, 
 as a kind of lift of $\gamma'$, a $q_{n+1}$-good Brouwer curve $\gamma$ for $f$.    By \eqref{E:bound Phi}, the associated fundamental 
domain $(f^{q_{n}}(\gamma),\gamma)$ has admissible coordinates $\Phi$, with respect to $f^{-q_n}$, satisfying 
\begin{equation}\label{E:Brouwer Phi estimates}
		 \|\Phi\|_{\diff^r([0,1]^2)} \leq G_r\big(\|\Phi'\|_{\diff^r([0,1]^2)},\, \|H\|_{\diff^r([0,1]^2)},\, \|f\|_{\diff^r(\A)},\,q_{n-1}\big)
\end{equation}
for each $r\in\N$.     Using \eqref{T2:H estimates 3}, \eqref{T2:H estimates 2} and \eqref{E:estimates renorm I} 
we see that $\|\Phi'\|_{\diff^r([0,1]^2)}$ and $\|H\|_{\diff^r([0,1]^2)}$ can be bounded in terms of $1/\a_n,q_{n-1}, q_{n}$ and 
$M$.    
The former can all be bounded in terms of $q_{n+1}$ using basic properties and \eqref{eq anupperlower}.    We arrive at a 
bound in the form of \eqref{E:bound W}, as desired.  

\smallskip

{\bf Step 4:}  
Retracing our steps, we assumed $\a$ satisfies \eqref{E:estimates fn+1 I} and \eqref{E:estimates fn+1 II}.    
Using \eqref{E:estimates renorm I} these assumptions take the form:  
\begin{align*}
								\beta_{n-1} \leq \varepsilon_0(1/\|f\|^{q_{n-1}}_{\diff^1([0,1]^2)}),\qquad 
								\G(\beta_{n-1})  \leq  \varepsilon_1\big(\beta_{n-1},\, 1/\|f\|^{q_{n-1}}_{\diff^1([0,1]^2)}\big)
								\intertext{and}
								\a_{n}  \leq  \varepsilon'\big(1/\|f\|_{\diff^{4}},1/q_{n}\big),\qquad 
								\G(\a_{n})  \leq \varepsilon_r\big(\a_{n},\, 1/\|f\|_{\diff^{5}},1/q_{n}\big).\qquad 
\end{align*}
where $\varepsilon', \varepsilon_r$ are increasing positive functions, with $\varepsilon$ independent of everything and $\varepsilon_r$ dependent only on $r$.  
%
Using \eqref{E:Gauss of alpha}, \eqref{E:Gauss of beta}, \eqref{eq continuefraction}, \eqref{eq anupperlower}, the above conditions are fulfilled if 
\ary\label{E:brouwer final conditions 2}
\begin{aligned}
								\hspace{100pt}\frac{1}{q_{n}} &\leq \varepsilon_0(1/M^{q_{n-1}})\\
								\frac{2q_{n-1}q_{n}}{q_{n+1}} & \leq  \varepsilon_1\left(\frac{1}{2q_{n}},\, 1/M^{q_{n-1}}\right)\\
								\frac{2q_{n}}{q_{n+1}} 	& \leq \varepsilon'\left(1/M,1/q_{n}\right)\\
								\frac{2q_{n+1}}{q_{n+2}} 	& \leq  \varepsilon_r\left(\frac{q_n}{2q_{n+1}},\, 1/M,1/q_{n}\right).  
\end{aligned}
\eary
In this final list; the first condition holds if $q_{n}$ is sufficiently large compared to $M$ and $q_{n-1}$.   
The second, third and fifth conditions hold if $q_{n+1}$ is large compared to $q_{n}$ and $M$, respectively compared to $q_n$, while the 
fourth condition holds if $q_{n+2}$ is large compared to $q_{n+1}$ and $M$.    
Thus \eqref{E:brouwer final conditions 2} holds if 
\aryst
			q_{n}\geq P(q_{n-1}, M), \qquad
								q_{n+1}\geq P'(q_{n}, M), \qquad 
								q_{n+2}\geq P_r(q_{n+1} , M)  
\earyst 
for some increasing functions $P, P', P_r$.    Replacing $P, P'$ and $P_r$ by their maximum we arrive at the 
conditions in \eqref{eq qnqn+1qn+2}.    This proves Theorem \ref{cor smooth return domain new version}. 
\end{proof}

\section{Construction of approximants}\label{S:approximants}
In this final section we use the Brouwer curve from Theorem \ref{cor smooth return domain new version} to construct periodic approximations to 
a pseudo-rotation $f$ with rotation number $\a:=\rho(f)$ under a generic condition on $\a$, as outlined in Introduction \ref{S:overview}.   
%
These goals are mostly contained in the following Theorem, whose proof will occupy most of the section.

Recall the following notation from Lemma \ref{lem construct rotation nbs}:  For any $P:\N\to\R_+$ define 
\aryst
			\cC_{P}(n) := \Big\{ x \in (0,1)\setminus \Q\ \Big|\ \ q_{n} > P(q_{n-1}),\quad q_{n+1} > P(q_{n}),\quad  q_{n+2} > P(q_{n+1})  \Big\},
\earyst 
where $(q_j)_{j\in\N}$ is the sequence of denominators in the best rational approximations to $x$, see Section \ref{S:rational}.  
\begin{thm}\label{thm. open dense} 
	 For each $(r,M,\epsilon) \in \N_{\geq 3} \times \N\times(0,1]$, there exists a function $P=P_{(r,M,\epsilon)}:\N\to\R_+$ 
	 such that the following holds.  
	  If $f:\A\to\A$ is any smooth pseudo-rotation with\footnote{We use bounds on one more derivative than in Theorem \ref{cor smooth return domain new version}, so that in \eqref{phi bounds} we get 
	  bounds on admissable coordinates up to order $r+1$.    This one extra derivative on the admissible coordinates is used to apply Lemma \ref{L:composition} in the proof of Proposition \ref{P:sigma small}.} 
	  \aryst
	  			\norm{f}_{\diff^{r+6}(\A)} \leq M, 
	  \earyst
	  and rotation number $\a=\rho(f)$ satisfying
	\aryst\label{D: A*}
		\a \in \cC_{P}(n)
	\earyst
	for some $n\in\N_{\mathrm{odd}}$, then there exists $h\in\diff^\infty(\A,\omega)$ so that 
	\ary\label{epsilon close}
	d_{\diff^{r - 2}(\A)}(f,\ h R_{\a}  h^{-1})\leq \epsilon. 
	\eary
\end{thm}


Assuming Theorem \ref{thm. open dense} let us prove the main result of the article: 

\begin{proof}[Proof of Theorem \ref{main thm F = closure O}]
Fix $(r,M,\epsilon) \in \N_{\geq 2} \times \N\times(0,1]$ and let $P_{(r,M,\epsilon)}:\N\to\R_+$ be the function from Theorem \ref{thm. open dense}.   Applying Lemma \ref{lem construct rotation nbs} to $P_{(r,M,\epsilon)}$ gives 
\aryst
				\cA(r, M, \epsilon):=\bigcup_{n\in\N_{\mathrm{odd}}} \cC_{P_{(r,M,\epsilon)}}(n)
\earyst
is open and dense in $(0,1)\setminus \Q$.    Thus the countable intersection
	\aryst
	\cA := \bigcap_{(r,M,k)\in\N_{\geq 2}\times\N\times\N}\cA(r, M, 1/k)
	\earyst
is a residual subset of $(0,1)\setminus \Q$ by Baire's theorem.   

We claim that this set $\cA$ does the job for Theorem \ref{main thm F = closure O}.   Indeed, 
fix $\a\in\cA$.   As remarked in the introduction it suffices to prove \eqref{F subset closure O}, namely that 
	\ary \label{eq finoclosure}
	F^{\infty}_\A(\a) \subset \overline{O^\infty_\A(\a)}
	\eary
where the closure is taken in the $C^\infty$-topology.   To show this inclusion, fix $f \in F^\infty_{\A}(\a)$ arbitrary.     
Then for any $\epsilon > 0$ and $r\in\N_{r\geq2}$  choose $M\in\N$ so that $\norm{f}_{\diff^{r+6}(\A)} < M$ and choose $k\in\N$ with $k^{-1}<\epsilon$.     
	Then since $\a \in \cA(r, M, 1/k)$, by Theorem \ref{thm. open dense} there exists $h \in \diff^{\infty}(\A, \omega)$ so that 
	$d_{\diff^{r-2}(\A)}(f, hR_\a h^{-1}) \leq 1/k < \epsilon$.  
	Since $\epsilon>0$ is arbitrary, this gives $F^{\infty}_\A(\a) \subset \overline{O^\infty_\A(\a)}$ with the closure in the $\diff^{r-2}$-topology.   Since $r$ is arbitrary we conclude \eqref{eq finoclosure} holds with closure in the $\diff^\infty$-topology as required.   
\end{proof} 

It remains to prove Theorem \ref{thm. open dense}, which occupies the rest of this section.   

\begin{proof}
Fix a tuple $(r_0,M, \epsilon_0)\in\N_{\geq 3}\times \N \times (0,1]$.   Let 
\ary\label{E:P and W}
		P_{(r_0,M)}:\N\to\R_+,\qquad W_{(r_0,M)}:\N\to\R_+,
\eary
be the increasing functions from Theorem \ref{cor smooth return domain new version} associated to $(r_0,M)$.    We will 
construct $P_{(r_0,M,\epsilon_0)}:\N\to\N$, as claimed in Theorem \ref{thm. open dense}, so that 
$P_{(r_0,M,\epsilon_0)}\ \geq\ P_{(r_0,M)}$.  
To this end, we fix a smooth pseudo-rotation $f:\A\to\A$ satisfying 
\ary\label{E:M bounds}
					\norm{f}_{\diff^{r_0+6}(\A)} \leq M    
\eary
and with rotation number $\a$.    To prove Theorem \ref{thm. open dense} it suffices to show the following:  That there exists a function 
$S=S_{(r_0,M,\epsilon_0)}:\N\to\N$, depending only on $(r_0,M,\epsilon_0)$ (rather than on $f$ specifically), so that if 
\ary\label{condition on a 1}
			\a\in\cC_{P_{(r_0,M)}}(n) \cap\Big\{ x\ \Big|\ q_{n+2} > S(q_{n+1}) \Big\} 
\eary
for some $n\in\N_{\mathrm{odd}}$, then there exists an $h\in\diff^{\infty}(\A,\omega)$ so that 
\ary\label{E:goal}
							d_{\diff^{r_0-2}(\A)}(f,h^{-1}R_{\a}h)\ <\ \epsilon_0.  
\eary
Indeed, then the theorem will be proven by setting $P_{(r_0,M,\epsilon_0)}:=\max\{ P_{(r_0,M)},\ S_{(r_0,M,\epsilon_0)}\}$. 
Note that condition \eqref{condition on a 1} is just saying that $\a\in\cC_{P_{(r_0,M)}}(n)$ and that $q_{n+2}$ is sufficiently large 
depending on $q_{n+1}$, $r_0, M, \epsilon_0$.   Therefore, let us assume that 
\ary\label{E:a condition P}
			\a\in\cC_{P_{(r_0,M)}}(n) 
\eary
for some $n\in\N_{\mathrm{odd}}$, which is a condition on $q_{n-1}, \ldots q_{n+2}$ (we abbrieviate $q_j=q_j(\a)$ for all $j\in\N$), 
and over the next pages, in the course of constructing the diffeomorphism $h$, we will several times require that $q_{n+2}$ is sufficiently large, in a sense that depends only on $(q_{n+1},r_0,M,\epsilon_0)$.    This implicitly leads to a condition of the form described by \eqref{condition on a 1}.  

\subsection{A decomposition of $\A$ that is almost permuted by $f$}\label{S:decomposition}
From condition \eqref{E:a condition P} we may apply Theorem \ref{cor smooth return domain new version} to $f$.   This has two 
important consequences.   The first is that $f$ has a smooth Brouwer curve $\gamma\subset\A$ for which the curves 
\ary\label{boundary curves 0}
				\Gamma:= \Big\{\ \gamma,\ f(\gamma), \ldots,\ f^{q_{n+1}-1}(\gamma)\ \Big\}
\eary 
are mutually disjoint in $\A$.     The second consequence of Theorem \ref{cor smooth return domain new version} is the existence of certain admissible 
coordinates for the region bounded by $f^{q_n}(\gamma)$ and $\gamma$.   This latter fact will not be used until the proofs of Proposition \ref{P:sigma small} and Lemma 
\ref{L:h with bounds}, so we will recall this in section \ref{S:periodic approx}.    

Returning to the curves in $\Gamma$, we would not expect the next iterate of $\gamma$ to be disjoint from the others; indeed, by Theorem \ref{thm aflxz} we expect $f^{q_{n+1}}$ to be close to the the identity, and therefore $f^{q_{n+1}}(\gamma)$ to be close to $\gamma$.    So one can think of $f$ as almost permuting the curves in $\Gamma$.  
Our first step will be to modify $f$ to make it genuinely permute the curves in $\Gamma$, that is, 
so that it maps the ''final curve'' $f^{q_{n+1}-1}(\gamma)$ in $\Gamma$ ``back'' to $\gamma$, closing up the cycle.     
If the changes are made in a small neighborhood of $f^{q_{n+1}-1}(\gamma)$, then the modified map $g:\A\to\A$ will agree with $f$ on all the other curves in $\Gamma$.   

We begin with a combinatorial statement: 

\begin{lemma}\label{L:closest curves}
If the curves in $\Gamma$ are mutually disjoint, then the two curves in $\Gamma$ that are closest to $\gamma$ are $f^{q_{n}}(\gamma)$ and 
$f^{q_{n+1}-q_{n}}(\gamma)$.     If $n\in\N_{\mathrm{odd}}$ then 
\aryst
		f^{q_{n}}(\gamma)\ <\ \gamma\ <\ f^{q_{n+1}-q_{n}}(\gamma)
\earyst
meaning that these curves meet the circle $B_1\subset\partial\A$ as a positively ordered triple.
\end{lemma}
\begin{proof} 
Since the restriction of $f$ to $B_1$ is a circle diffeomorphism with rotation number $\a$, the claim follows from standard combinatorial facts about the dynamics of circle 
homeomorphisms.  
\end{proof}
We apply this Lemma without justification, as all conditions are satisfied.  
\vspace{-5pt} 
	\begin{figure}[h]
        \centering
        \scalebox{0.6}{
\begingroup%
  \makeatletter%
  \providecommand\color[2][]{%
    \errmessage{(Inkscape) Color is used for the text in Inkscape, but the package 'color.sty' is not loaded}%
    \renewcommand\color[2][]{}%
  }%
  \providecommand\transparent[1]{%
    \errmessage{(Inkscape) Transparency is used (non-zero) for the text in Inkscape, but the package 'transparent.sty' is not loaded}%
    \renewcommand\transparent[1]{}%
  }%
  \providecommand\rotatebox[2]{#2}%
  \newcommand*\fsize{\dimexpr\f@size pt\relax}%
  \newcommand*\lineheight[1]{\fontsize{\fsize}{#1\fsize}\selectfont}%
  \ifx\svgwidth\undefined%
    \setlength{\unitlength}{349.0232244bp}%
    \ifx\svgscale\undefined%
      \relax%
    \else%
      \setlength{\unitlength}{\unitlength * \real{\svgscale}}%
    \fi%
  \else%
    \setlength{\unitlength}{\svgwidth}%
  \fi%
  \global\let\svgwidth\undefined%
  \global\let\svgscale\undefined%
  \makeatother%
  \begin{picture}(1,0.47077331)%
    \lineheight{1}%
    \setlength\tabcolsep{0pt}%
    \put(0,0){\includegraphics[width=\unitlength,page=1]{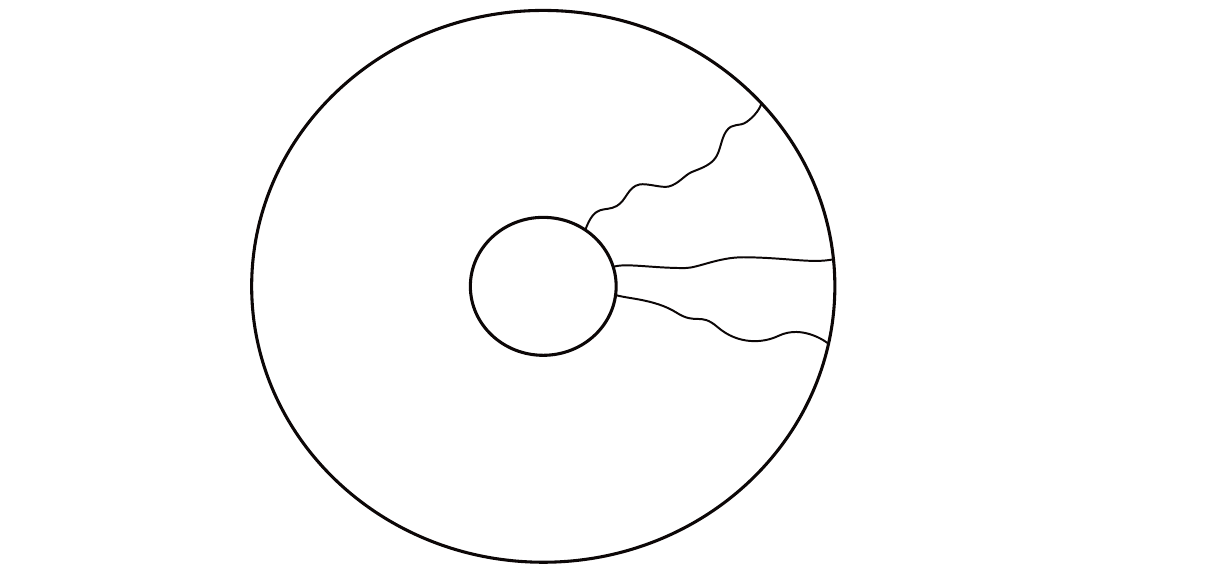}}%
    \put(0.6889969,0.17143175){\color[rgb]{0,0,0}\makebox(0,0)[lt]{\lineheight{1.25}\smash{\begin{tabular}[t]{l}$f^{q_n}(\gamma)$\end{tabular}}}}%
    \put(0.69650152,0.24848776){\color[rgb]{0,0,0}\makebox(0,0)[lt]{\lineheight{1.25}\smash{\begin{tabular}[t]{l}$\gamma$\end{tabular}}}}%
    \put(0,0){\includegraphics[width=\unitlength,page=2]{figure-4.pdf}}%
    \put(0.63289124,0.39339066){\color[rgb]{0,0,0}\makebox(0,0)[lt]{\lineheight{1.25}\smash{\begin{tabular}[t]{l}$f^{q_{n+1}-q_n}(\gamma)$\end{tabular}}}}%
    \put(0,0){\includegraphics[width=\unitlength,page=3]{figure-4.pdf}}%
    \put(0.69650152,0.24848776){\color[rgb]{0,0,0}\makebox(0,0)[lt]{\lineheight{1.25}\smash{\begin{tabular}[t]{l}$\gamma$\end{tabular}}}}%
    \put(0.69305896,0.27735249){\color[rgb]{1,0,0}\makebox(0,0)[lt]{\lineheight{1.25}\smash{\begin{tabular}[t]{l}$f^{q_{n+1}}(\gamma)$\end{tabular}}}}%
  \end{picture}%
\endgroup%

        }
	\vspace{-5pt}
        \caption{The radial curves in black illustrate $\Gamma$, see \eqref{boundary curves 0}, which determine a tiling $\cT_0,\ldots,\cT_{q_{n+1}-1}$ in 
        Definition \ref{D:tiling}.   
        The red curve $\gamma_*:=f^{q_{n+1}}(\gamma)$ 
        potentially intersects $\gamma$.   After modifying $f$ (to $g$ as in 
        \eqref{D:g}) the red curve coincides perfectly with $\gamma$ and the curves in $\Gamma$ remain unchanged.}
        \label{F:f to g}
        \end{figure}
\begin{defi}\label{D:Omega}
Let $\Omega\subset\A$ denote the closed region 
\aryst
			\Omega:= \big(f^{q_{n}}(\gamma),f^{q_{n+1}-q_{n}}(\gamma)\big), 
\earyst
and $\cB,\cC\subset\A$ be the following decomposition into closed subsets:
\ary\label{E:B and C} 
		\cB:=\big(f^{q_{n}}(\gamma),\gamma\big),\qquad \cC:=\big(\gamma,f^{q_{n+1}-q_{n}}(\gamma)\big)
\eary
with disjoint interiors.  
\end{defi}
%
%

\begin{lemma}\label{L:gamma* in Omega}
$f^{q_{n+1}}(\gamma)\subset\Int(\Omega)$.   
\end{lemma}
\begin{proof} 
By the theory of circle homeomorphisms $f^{q_{n+1}}(\gamma)$ meets $B_1\subset\partial\A$ at a point 
strictly closer to $x:=\gamma\cap B_1$ than $f^j(x)$ for all $0< j <q_{n+1}$ (for the combinatorial distance).      Thus $f^{q_{n+1}}(\gamma)$ has non-empty intersection with $\Int(\Omega)$.   
So it suffices to show 
\aryst
			f^{q_{n+1}}(\gamma)\cap f^{q_n}(\gamma)=\emptyset,\qquad f^{q_{n+1}}(\gamma)\cap f^{q_{n+1}-q_{n}}(\gamma)=\emptyset.  
\earyst 
The first condition follows by applying $f^{q_n}$ to $\gamma\cap f^{q_{n+1}-q_{n}}(\gamma)=\emptyset$, the second by applying $f^{q_{n+1}-q_n}$ to $f^{q_n}(\gamma)\cap\gamma=\emptyset$. 
\end{proof} 


Let us consider a diffeomorphism of the form 
\ary\label{D:g}
					g:=\sigma\circ f:\A\to\A
\eary
where $\sigma\in\diff^\infty(\A)$ satisfies 
\ary
		&\mathrm{supp}(\sigma)\subset\Int(\Omega) \label{sigma support}\\
		& \sigma\big(\gamma_*\big)=\gamma,     \label{sigma closing condition}
\eary
where $\mathrm{supp}(\sigma):=\overline{\{x\in\A | \sigma(x)\neq x\}}$ is the support of $\sigma$, and where 
\aryst
		\gamma_*:=f^{q_{n+1}}(\gamma).
\earyst
Any such $\sigma$ is isotopic to $\id_{\A}$, hence so is $g$.   

\begin{lemma}[$g$ permutes the curves in $\Gamma$]\label{L:sigma 1}
The map $g$ agrees with $f$ on $\A\backslash f^{-1}(\Int(\Omega))$ and permutes the curves in $\Gamma$ cyclically: 
\begin{enumerate}
 \item $g^j(\gamma)=f^j(\gamma)$ for all $0\leq j\leq q_{n+1}-1$.    
 \item $g^{q_{n+1}}(\gamma)=\gamma$, so $g(\Gamma)=\Gamma$.   
 \item Hence the regions $\cB, \cC$ from \ref{E:B and C} are described in terms of $g$ as:  
 \ary\label{E:B and C in g}
 			\cB=\big(g^{q_{n}}(\gamma),\gamma\big),\qquad \cC=\big(\gamma,g^{q_{n+1}-q_{n}}(\gamma)\big).
 \eary
\end{enumerate}
\end{lemma}
\begin{proof} 
The claim that $g=f$ on $\A\backslash f^{-1}(\Int(\Omega))$ is because $\sigma$ is supported in $\Int(\Omega)$ by \eqref{sigma support}.

Hence to see item (1):  For each $0\leq j\leq q_{n+1}-1$ it suffices to know that $f^{j-1}(\gamma)\subset \A\backslash f^{-1}(\Int(\Omega))$.    
Equivalently, that $f^{j}(\gamma)\subset \A\backslash \Int(\Omega)$.   Equivalently that $\Gamma\subset \A\backslash \Int(\Omega)$, which is 
Lemma \ref{L:closest curves}.  

%
%

Item (2) is then: $g^{q_{n+1}}(\gamma)=g\circ g^{q_{n+1}-1}(\gamma)=g\circ f^{q_{n+1}-1}(\gamma)$ from item (1), followed by 
$g\circ f^{q_{n+1}-1}(\gamma)=\sigma\circ f\circ f^{q_{n+1}-1}(\gamma)=\sigma(\gamma_*)=\gamma$ by \eqref{sigma closing condition}. 

Item (3) is immediate from Item (1).
\end{proof}

\begin{lemma}\label{L:sigma 2}
Let $\cT\subset\A$ be the closure of any component of the complement $\A\backslash\Gamma$.    Then  
$\cT$ is a fundamental domain for $g$ in the following sense: 
 \ary\label{decomp}
 		\cT \cup g(\cT) \cup \ldots \cup g^{q_{n+1}-1}(\cT) = \A 
\eary
with mutually disjoint interiors.   
\end{lemma}
\begin{proof} 
By Lemma \ref{L:sigma 1} $\A\backslash\Gamma$ is $g$-invariant and $g$ induces a bijection on the set of 
connected components of $\A\backslash\Gamma$.  
Let $U$ be the interior of the region bounded by $(f^{q_n}(\gamma),\gamma)$.     Then if $U\cup g(U)\cup\cdots\cup g^{q_{n+1}-1}(U)$ 
is not equal to $\A\backslash\Gamma$ then there exists $k\in \{1,\ldots,q_{n+1}-1\}$ so that $g^k(U)=U$.   
Comparing the right sides of these two regions gives $g^k(\gamma)=\gamma$, since $g$ is isotopic to the identity, hence 
$f^k(\gamma)=\gamma$ by Lemma \ref{L:sigma 1}, contradicting the disjointness of the curves in $\Gamma$.    
Thus $U\cup g(U)\cup\cdots\cup g^{q_{n+1}-1}(U)=\A\backslash\Gamma$, and the Lemma follows.  
\end{proof}
\begin{defi}[Tiling of $\A$]\label{D:tiling}
Using cyclic indices set 
\ary\label{tiling}
			\cT_j:=g^j(\cB)\qquad\forall j\in\Z_{q_{n+1}}.
\eary  
\end{defi}

By Lemma \ref{L:sigma 2}, $\A=\bigcup_{j\in\Z_{q_{n+1}}}\cT_j$ 
is a decomposition of $\A$ into topological rectangles whose interiors are mutually disjoint, and $\{\Int(\cT_0),\ldots,\Int(\cT_{q_{n+1}})\}$ are the components  
of $\A\backslash\Gamma$.    Note that $\cB=(g^{q_n}(\gamma),\gamma)$ implies $\cT_{q_{n+1}-q_{n}}:=g^{q_{n+1}-q_{n}}(\cB)=(g^{q_{n+1}}(\gamma),g^{q_{n+1}-q_{n}}(\gamma))=(\gamma,g^{q_{n+1}-q_{n}}(\gamma))=\cC$.   Thus 
\ary\label{cC as R}
				\cT_0=\cB,\qquad \cT_{q_{n+1}-q_{n}}=\cC.   
\eary
From \eqref{tiling}, $g^{q_{n+1}-q_{n}}$ restricts to a diffeomorphism from $\cT_0$ to $\cT_{q_{n+1}-q_{n}}$, i.e.\ from $\cB$ to $\cC$.   
Hence, as $g^{q_{n+1}-q_{n}}=g^{-q_{n}}$, the inverse map $g^{q_n}$ restricts to a diffeomorphism 
\ary
				g^{q_{n}}:\cC\to\cB.  
\eary 
%
%
Finally, it follows easily from Lemma \ref{L:sigma 1} that
\ary\label{L:g}
			g|_{\cT_j}\ =\ \left\{\begin{aligned}
						& \sigma\circ f &   &  \mbox{ if }\  j\in\{-1,-q_n-1\}\\
						& f   & 			& \mbox{ otherwise.}\\
						\end{aligned}\right.
\eary

\subsection{A $q_{n+1}$-periodic deformation}\label{S:g periodic}
In the previous section we saw how to deform $f$ to a diffeomorphism of the form $g:=\sigma\circ f$ 
so that $g$ cyclically permutes the curves in $\Gamma$.  
In this section we state further conditions on $\sigma$ that ensure 
$g$ will be periodic. 

Recall $\Omega=\cB\cup\cC$ is a decomposition along the curve $\gamma$.   By Lemma \ref{L:gamma* in Omega} there is also a 
decomposition along $\gamma_*=f^{q_{n+1}}(\gamma)$ denoted by  
\aryst 
		\Omega=\cB'\cup\cC',        
\earyst 
where $\cB',\cC'\subset\A$ are the closed regions bounded by $\big(f^{q_{n}}(\gamma),\gamma_*\big)=\big(g^{q_{n}}(\gamma),\gamma_*\big)$ 
and $\big(\gamma_*,f^{q_{n+1}-q_{n}}(\gamma)\big)=\big(\gamma_*,g^{q_{n+1}-q_{n}}(\gamma)\big)$ respectively.   
By \eqref{sigma closing condition} $\sigma$ maps $\cB', \cC'$ to $\cB, \cC$ respectively.   Thus $\sigma$ restricts to a pair of diffeomorphisms 
\ary\label{D:L and R}
				\sigma_L:\cB'\to\cB, \qquad \sigma_R:\cC'\to\cC. 
\eary
\vspace{-12pt}
\begin{figure}[h]
        \centering
        \scalebox{0.75}{
        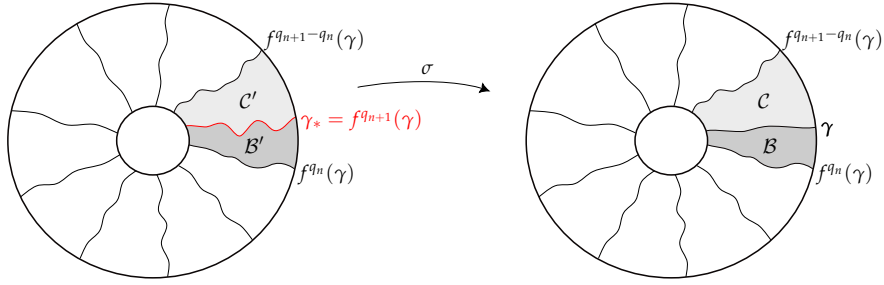
        }
        \caption{Illustrating the deformation $\sigma:\A\to\A$, which is supported on the shaded regions, see \eqref{D:L and R}.} 
        \label{F:f to g 2}
\end{figure}

\noindent In terms of $\sigma_L$ and $\sigma_R$, \eqref{L:g} can be refined to: 
\ary\label{L:g 2}
			g|_{\cT_j}\ =\ \left\{\begin{aligned}
						& f   & 			& \mbox{ if }\ 0\leq j\leq q_{n+1}-q_n-2\\
						& \sigma_R\circ f &   &  \mbox{ if }\  j= q_{n+1}-q_n-1\\
						& f   & 			& \mbox{ if }\ q_{n+1}-q_n\leq j\leq q_{n+1}-2\\
						& \sigma_L\circ f &   &  \mbox{ if }\  j= q_{n+1}-1.
						\end{aligned}\right.
\eary
Every $g$-orbit passes through $\cT_0$.   For $x\in\cT_0$, applying \eqref{L:g 2} successively to $g^j(x)\in\cT_j$ we find that 
$g^{q_{n+1}}=\id_{\A}$ if and only if 
\ary\label{E:R from L}
					\sigma^{-1}_L\ =\ f^{q_n}\circ \sigma_R\circ f^{q_{n+1}-q_{n}}|_{\cB}.  
\eary  
\begin{lemma}\label{L:sigma periodic 2}
Let $\sigma_L:\cB'\to\cB$ be a smooth diffeomorphism satisfying 
\ary\label{sigmaL conditions}
					\sigma_L=\left\{\begin{aligned}
							& \id \qquad\quad\ \  \mbox{near to the left side of }\  \cB' \\
							& f^{-q_{n+1}}\qquad \mbox{near to the right side of  }\  \cB'.   
							\end{aligned}\right.
\eary
Then if we set  
\ary\label{E:R from L 2}
					\sigma^{-1}_R\ :=\ f^{q_{n+1}-q_{n}}\circ \sigma_L\circ f^{q_{n}}|_{\cC}
\eary 
then the map $\sigma:\A\to\A$ defined to be $\sigma_L$ on $\cB'$, to be $\sigma_R$ on $\cC'$, and to be the identity on $\A\backslash\Omega$, 
defines a smooth diffeomorphism on $\A$ that satisfies \eqref{sigma support} and \eqref{sigma closing condition}, 
and $g:=\sigma\circ f$ satisfies $g^{q_{n+1}}=\id_{\A}$.  
\end{lemma}
\begin{proof}
It suffices to check that $\sigma_R$ defined by \eqref{E:R from L 2} is the identity near to the right end of $\cC'$ and coincides with $f^{-q_{n+1}}$ near to the left end of $\cC'$. 
Then it follows that $\sigma_L$ and $\sigma_R$ both equal $f^{-q_{n+1}}$ near to the shared boundary $\gamma_*$.    Using \eqref{E:R from L} 
the claim $g^{q_{n+1}}=\id_{\A}$ is easy to verify.   
\end{proof}

\subsection{A $q_{n+1}$-periodic approximation}\label{S:periodic approx}
In the last section we gave conditions on the deformation $\sigma$ that ensure $g:=\sigma\circ f$ 
is periodic; $g^{q_{n+1}}=\id_{\A}$.   

In this section we show that $\sigma$ can also be made ''small'', so that $g$ 
is close to $f$.   This requires a further condition on the rotation number of $f$:  

\begin{prop}\label{P:sigma small}
Given $\epsilon>0$, if $q_{n+2}$ is sufficiently large depending on $q_{n+1}$ and $(r_0,M,\epsilon)$, there exists $\sigma\in\diff^\infty(\A)$ so that the deformation 
$g:=\sigma\circ f :\A\to\A$ satisfies $g^{q_{n+1}}=\id_{\A}$ and 
\aryst
				d_{\diff^{r_0}(\A)}(\sigma,\id)<\epsilon.  
\earyst
\end{prop}
By Lemma \ref{L:sigma periodic 2} it suffices to construct $\sigma_L:\cB'\to\cB$ satisfying \eqref{sigmaL conditions} and sufficiently close to the identity.   
It is convenient to transfer the construction from $\B'$, which may degenerate, to the fixed region $[0,1]^2$, via the change of coordinates provided by 
Theorem \ref{cor smooth return domain new version}, before transferring back to $\B'$. 

As preparation for the proof of Proposition \ref{P:sigma small}, recall by Theorem \ref{cor smooth return domain new version}
there exists a smooth diffeomorphism $\phi:[0,1]^2\to\cB$ satisfying 
\ary\label{phi bounds}
			\|\phi\|_{\diff^{r_0+1}([0,1]^2)}\ \leq\ W_{(r_0+1,M)}(q_{n+1}), 
\eary 
where $W_{(r_0,M)}$ is the function from \eqref{E:P and W},    
and $\phi$ has constant Jacobian and 
extends smoothly to $\tilde\A$, which we still denote by $\phi$, via the relation 
\ary\label{phi T}
		\phi\circ T=f^{-q_n}\circ\phi.
\eary    
This extended map, which also has constant Jacobian, admits uniform bounds on compact subsets.   In particular  
\ary\label{E:phi bound 2}
					\|\phi\|_{\diff^{r_0+1}([-1,2]\times[0,1])}\ \leq\ \hat{W}_{(r_0,M)}(q_{n+1}), 
\eary
for some increasing $\hat{W}_{(r_0,M)}:\N\to\R_+$.   Now we prove the Proposition. 

\begin{proof} Fix $\epsilon>0$ and $(r_0,M,q_{n+1})$.  

\textbf{Step 1:} By Theorem \ref{thm aflxz},  Lemma \ref{L:composition} and  \eqref{E:M bounds},   if $q_{n+2}$ is sufficiently large depending on $(r_0,M,q_{n+1},\epsilon)$,
there exists $\delta=\delta(r_0,M,q_{n+1},\epsilon)>0$ so that
$d_{\diff^{r_0}}(\sigma_L,\id) < \delta$ implies $d_{\diff^{r_0}(\cC,\cC')}(f^{q_{n+1}-q_n}\circ\sigma_L\circ f^{q_n},\id) < \epsilon$ 
for all $\sigma_L\in \diff^{r_0}(\cB',\cB)$.   Here we used that both factors $f^{q_{n+1}-q_n}$ and $f^{q_n}$ are uniformly bounded in $\diff^{r_0+1}(\A)$ by 
a constant depending on $M$ and $q_{n+1}$.   
Thus by Lemma \ref{L:sigma periodic 2} it suffices to construct a map $\sigma_L:\cB'\to\cB$ satisfying \eqref{sigmaL conditions} with $d_{\diff^{r_0}}(\sigma_L,\id)$ 
small.   More precisely, so that for all $\epsilon_0>0$, if $q_{n+2}$ is sufficiently large depending on $(r_0,M,q_{n+1},\epsilon_0)$, then 
\ary\label{bound sigmaL}
		 				d_{\diff^{r_0}(\cB',\cB)}(\sigma_L,\id) < \epsilon_0.    
\eary
%
%
%
%

\smallskip

\textbf{Step 2:} 
We claim that if $q_{n+2}$ is sufficiently large depending on $M$, then $\gamma_*=f^{q_{n+1}}(\gamma)$ is sufficiently close to 
$\gamma=\phi(\{1\}\times[0,1])$ that 
\ary\label{E:gamma* condition}
							\gamma_*\subset\phi\big((3/4,5/4)\times[0,1]\big).   
\eary
Indeed, let $\cU\subset\A$ be the $(4M)^{-1}$-neighborhood of $\gamma$.   Then $\|D\phi^{-1}\|\leq M$ implies 
$\cU\subset \phi\big((3/4,5/4)\times[0,1]\big)$.   By Theorem \ref{thm aflxz} $d_{H}(\gamma_*,\gamma)<(4M)^{-1}$ 
 if $q_{n+2}$ is sufficiently large, which proves $\gamma_*\subset \cU$ and hence the claim.  Indeed
\aryst
		d_{H}(\gamma_*,\gamma)=d_{H}(f^{q_{n+1}}(\gamma),\gamma)\leq d_{C^0}(f^{q_{n+1}},\id)\leq A_0(\|\rho(f^{q_{n+1}})\|_{\R/\Z}, M^{q_{n+1}}), 
\earyst
and the right hand side can be made as small as we wish if $q_{n+2}$ is sufficiently large depending on $q_{n+1}$ and $M$, 
because $\|\rho(f^{q_{n+1}})\|_{\R/\Z}=\b_{n+1}<1/q_{n+2}$ by \eqref{eq continuefraction}.   
%
By \eqref{E:gamma* condition} we may define the curve
\aryst
								\gamma':=\phi^{-1}(\gamma_*), 
\earyst 
and this satisfies 
\ary\label{bound gamma'}
							\gamma'\subset(3/4,5/4)\times[0,1].
\eary

\smallskip

\textbf{Step 3:} Due to step 2, if $q_{n+2}$ is sufficiently large depending on $q_{n+1}$ and $M$, then the closed region 
$\cR := \big(\{0\}\times[0,1],\,\gamma'\big) \subset\tilde\A$
bounded between $\{0\}\times[0,1]$ and $\gamma'$, is well defined and the restriction of $\phi$ to $\cR$ is a chart for $\cB'$: Indeed
\ary\label{coords for B'}
					\phi:\cR\to\cB' 
\eary
is a diffeomorphism, because $\cB'$ has the same left side as $\cB$, and its right side is $\gamma_*$.  
Note that $\cR$ satisfies the uniform ''non-degeneracy'' condition 
\ary\label{E:nondeg R}
						[0,3/4]\times[0,1]\subset\cR\subset [0,5/4]\times[0,1]  
\eary 
independent of all the maps involved.  

\smallskip

\textbf{Step 4:} We wish to construct a map on $\cR$ that interpolates between the identity near the left side of $\cR$ and the map 
$\phi^{-1}\circ f^{-q_{n+1}}\circ\phi$ near the right side.   This is immediate from the following
claim: For all $\epsilon_2>0$, there exists $\delta_2=\delta_2(r,\epsilon_2)>0$, so that if $\psi:\cR\to\tilde\A$ is any diffeomorphism onto its image 
with $d_{\diff^{r_0}(\cR,\psi(\cR))}(\psi,\id) < \delta_2$, 
then there exists a smooth diffeomorphism onto its image $\hat{\psi}:\cR\to\tilde\A$ satisfying $d_{\diff^{r_0}(\cR,\hat{\psi}(\cR))}(\hat{\psi},\id) < \epsilon_2$ 
and so that $\hat{\psi}=\id$ on $\big((-\infty,1/4]\times[0,1]\big)$ and $\hat{\psi}=\psi$ on $\big([1/2,\infty)\times[0,1]\big)$.
This claim can be proven easily by writing $\psi$ as an interpolation of $\id$ and $\psi$ constructed via partition of unity by smooth cutt-off functions.
  A key point is that the cutt-off functions can be fixed because of our uniform 
control of the domain as expressed in \eqref{E:nondeg R}.    
%

\medskip

\textbf{Step 5:} Let $\epsilon_3>0$.   We claim that if $q_{n+2}$ is sufficiently large depending on $(r_0,M,q_{n+1},\epsilon_3)$, then 
\ary\label{bound 101}
		 d_{\diff^{r_0}}(\phi^{-1}\circ f^{-q_{n+1}}\circ\phi,\id) < \epsilon_3.  
\eary
Indeed, by Lemma \ref{L:composition} and \eqref{E:phi bound 2}, there exists $\delta_3=\delta_3(r_0,M,q_{n+1},\epsilon_3)>0$ 
so that \eqref{bound 101} holds whenever $d_{\diff^{r_0}}(f^{-q_{n+1}},\id) < \delta_3$.   By Theorem \ref{thm aflxz}, 
$d_{\diff^{r_0}}(f^{-q_{n+1}},\id) < A_{r_0}(\|\rho(f^{q_{n+1}})\|_{\R/\Z}, M^{q_{n+1}})$.    The claim then follows using $\|\rho(f^{q_{n+1}})\|_{\R/\Z}=\b_{n+1}<1/q_{n+2}$ because of \eqref{eq continuefraction}.  

Now, if \eqref{bound 101} is sufficiently small then 
%
we can apply Step 4 to the map $\psi=\phi^{-1}\circ f^{-q_{n+1}}\circ\phi$ 
and conclude that there exists a smooth diffeomorphism onto its image $\hat{\psi}:\cR\to\tilde\A$, satisfying 
\ary\label{sigmahat conditions}
					\hat{\psi}=\left\{\begin{aligned}
							& \id \qquad\qquad\qquad\qquad \mbox{near to the left side of }\  \cR \\
							& \phi^{-1}\circ f^{-q_{n+1}}\circ\phi\qquad\  \mbox{near to the right side of  }\  \cR   
							\end{aligned}\right.
\eary
and so that for all $\epsilon_4>0$, if $q_{n+2}$ is sufficiently large depending also on $(r_0,M,q_{n+1})$, then 
\ary\label{bound psihat}
		 				d_{\diff^{r_0}}(\hat{\psi},\id)<\epsilon_4.    
\eary
%
%
%
Note that by \eqref{sigmahat conditions} $\hat{\psi}$ defines a diffeomorphism $\hat{\psi}:\cR\to [0,1]^2$.    
For example, $\hat{\psi}$ maps the right side of $\cR=\gamma'$ to $\phi^{-1}\circ f^{-q_{n+1}}\circ\phi(\gamma')$, which simplifies to 
$\phi^{-1}\circ f^{-q_{n+1}}(\gamma_*)=\phi^{-1}(\gamma)=\{1\}\times[0,1]$.  

\medskip

\textbf{Step 6:}  Finally we transfer the construction back to $\cB'$.     Set 
\aryst
					\sigma_L:=\phi\circ\hat{\psi}\circ\phi^{-1}:\cB'\to\cB,
\earyst
the domain and target follow from \eqref{coords for B'} and $\hat{\psi}(\cR)=[0,1]^2$.   Then \eqref{sigmaL conditions} follows from \eqref{sigmahat conditions}.   
By \eqref{bound psihat} and Lemma \ref{L:composition} and \eqref{E:phi bound 2}, we can make $d_{\diff^{r_0}}(\phi\circ\hat{\psi}\circ\phi^{-1},\id)$ 
as small as we wish, if $q_{n+2}$ is sufficiently large depending also on $(r_0,M,q_{n+1})$.   This is precisely \eqref{bound sigmaL}.
\end{proof}
%

\begin{cor}\label{L:periodic approximations}
%
Given $\epsilon>0$, if $q_{n+2}$ is sufficiently large depending on $(q_{n+1},r_0,M,\epsilon)$, there exists $g\in\diff^{\infty}(\A)$ 
so that $g^{q_{n+1}}=\id_{\A}$ and $d_{\diff^{r_0}(\A)}(f,g)<\epsilon$.   
\end{cor}
\begin{proof} 
From Lemma \ref{L:composition} and the bound \eqref{E:M bounds}, there exists $\delta>0$ depending on $(r_0,M,\epsilon)$, so that if $d_{\diff^{r_0}(\A)}(\sigma,\id)<\delta$ then 
$d_{\diff^{r_0}(\A)}(\sigma\circ f,f)<\epsilon$.   Now the statement follows from Proposition \ref{P:sigma small}.   
\end{proof}
%
%

\subsection{Conjugating $g$ to the rigid rotation $R_{p_{n+1}/q_{n+1}}$}\label{S:h}
In the previous sections we saw how to deform $f$ to a diffeomorphism of the form $g:=\sigma\circ f$ 
so that $g^{q_{n+1}}=\id_{\A}$ and $d_{\diff^{r_0}(\A)}(f,g)$ is small.  
It follows automatically from the periodicity of $g$ that there exists $h\in\diff^{\infty}(\A)$ so that 
\ary\label{E:conjugation 1}
		h\circ g\circ h^{-1}=R_{p_{n+1}/q_{n+1}}, 
\eary
where $p_{n+1}/q_{n+1}$ is the $n$th best rational approximation to the rotation number $\a=\rho(f)$, see \ref{S:rational}, and $R_{\theta}$ 
is the rigid rotation through angle $2\pi\theta$.  
Such an abstract existence result is not enough, as we need an $h$ with bounds in the $\diff^{r_0}$-sense, which is the goal of this section.  

To do this we introduce the tiling $\cT'_0,\ldots,\cT'_{q_{{n+1}}-1}$ of $\A$ for the rigid rotation $s:=R_{p_{n+1}/q_{n+1}}$, that is analogous to 
the tiling $\cT_0,\ldots,\cT_{q_{{n+1}}-1}$ introduced in \eqref{D:tiling} for $g$.     We fix a diffeomorphism $h_0:\cT_0\to\cT'_0$, and then define diffeomorphisms $h_j:\cT_j\to\cT'_j$ 
by setting $h_j:=s^j\circ h_0\circ g^{-j}$.   It is obvious that if this collection of maps match together smoothly then they define a smooth diffeomorphism $h:\A\to\A$ via 
$h|_{\cT_j}:=h_j$, which satisfies \eqref{E:conjugation 1} by construction.    To match smoothly we need to check that if $h_i$ and $h_j$ are defined on neighboring 
elements in the tiling, that they match up smoothly on the common side of their domains.    It is enough to check this matching for a single pair of maps with adjacent 
domains, for example we will do this for $h_0$ and $h_{q_{n+1}-q_{n}}$ which share the curve $\gamma$ in their respective domains.    Then by construction all other 
adjacent maps will also match smoothly.    The matching condition on $h_0$ that makes this work is established in Lemma \ref{L: match}.    
It remains then to find such an $h_0$ that additionally has certain $\diff^{r_0}$-bounds, this is the content of Lemma \ref{L:h with bounds}.  

To carry out this program it is convenient to abbrieviate 
\aryst
		s:\A\to\A, \quad s:=R_{p_{n+1}/q_{n+1}}.
\earyst
Set $\gamma':=\{0\}\times[0,1]\subset\A$.   Then the curves in 
\aryst
	\Gamma':=\big\{ \gamma',s(\gamma'),\ldots,s^{q_{n+1}-1}(\gamma') \big\}
\earyst
are mutually disjoint.   So from the definition of $q_n$, as $n$ is odd, and as $s$ is a rigid rotation, we know that the $q_n$-th iterate of $\gamma'$ is the closest 
curve in $\Gamma'$ to $\gamma'$ that lies to the left of $\gamma'$.    More precisely, the interior of 
$\cT'_0:=(s^{q_n}(\gamma'),\gamma')$ 
is a component of the complement $\A\backslash\Gamma'$, and the other components are uniquely indexed by setting: 
\aryst
				\cT'_j:=s^j(\cT'_0)\qquad \forall j\in\Z_{q_{n+1}}.  
\earyst
%
Applying similar reasoning to the tiling $\{\cT_j\}$, or using \eqref{cC as R} and $f^{q_n}(\gamma)=g^{q_n}(\gamma)$ from Lemma \ref{L:sigma 1}, we can write, 
\ary\label{E:gqn}
				\cT_0=(g^{q_n}(\gamma),\gamma).   
\eary
Thus $g^{-q_n}(\cT_0)=\cT_{-q_n}$ and $\cT_0$ are the two elements in the tiling of Definition \ref{D:tiling} that share $\gamma$ as a side.   

\begin{lemma}\label{L: match}
If $h_0:\cT_0\to\cT'_0$ extends to a smooth diffeomorphism between open neighborhoods $\hat{\cT}_0\to \hat{\cT}'_0$ in $\A$, so that 
\ary\label{D:h_0 match} 
				h_0=s^{-q_{n}}\circ h_0\circ g^{q_n}
\eary
on $g^{-q_n}(\hat{\cT}_0)\cap\hat{\cT}_0$ or on a possibly smaller open neighborhood of $\gamma$, then
\ary\label{D:h}
	\qquad\qquad	h : \Int(\cT_j)\to \Int(\cT'_j),\quad h|_{\Int(\cT_j)}:=s^j\circ h_0\circ g^{-j}\quad \forall j\in\Z_{q_{n+1}}
\eary
defines a smooth diffeomorphism from $\A\backslash\Gamma$ to $\A\backslash\Gamma'$ which extends to a 
smooth diffeomorphism $h:\A\to\A$ that satisfies $h\circ g=s\circ h$.    
\end{lemma}

\begin{proof}\leavevmode

\textbf{Step 1:} Clearly \eqref{D:h} defines a smooth diffeomorphism 
$h:\A\backslash\Gamma\to \A\backslash\Gamma'$ satisfying
\ary\label{conjugation 1}
					h\circ g = s\circ h\qquad \mbox{on } \A\backslash\Gamma.
\eary
%

\smallskip 

\textbf{Step 2:} We show that $h:\A\backslash\Gamma\to \A\backslash\Gamma'$ extends smoothly over $\gamma$.   
Indeed, we observed from \eqref{E:gqn} that $\gamma$ is the side shared by $\cT_0$ and $\cT_{-q_n}$.   
By \eqref{D:h}, $h=h_0$ on $\Int(\cT_0)$ and $h=s^{-q_{n}}\circ h_0\circ g^{q_n}$ on $\Int(\cT_{-q_n})$.   Therefore by \eqref{D:h_0 match} $h$ agrees  
with the extension of $h_0:\hat{\cT}_0\to \hat{\cT}'_0$ on an open neighborhood $U\subset g^{-q_n}(\hat{\cT}_0)\cap\hat{\cT}_0\subset\A$ of $\gamma$.   
Thus $h=h_0$ on $U\backslash\{\gamma\}$ and so $h$ extends smoothly (equal to $h_0$) on all of $U$, 
which it maps diffeomorphically to the open neighborhood $U':=h_0(U)$ of $\gamma'$.   

\smallskip 

\textbf{Step 3:} Now we show that $h$ extends to an element of $\diff^{\infty}(\A)$.   
Indeed, for $j\in\Z_{q_{n+1}}$, by step 1 $h=s^j\circ h\circ g^{-j}$ on $\A\backslash\Gamma$ so also on $g^j(U)\backslash\Gamma=g^j(U)\backslash\{g^j(\gamma)\}$.   
So $h=s^j\circ h|_{U\backslash\{\gamma\}}\circ g^{-j}$ 
on $g^j(U)\backslash\{g^j(\gamma)\}$.    By Step 2 this extends smoothly to $s^j\circ h|_{U}\circ g^{-j}$.  
\end{proof}
It remains to construct a diffeomorphism $h_0:\cT_0\to\cT'_0$ satisfying the conditions of Lemma \ref{L: match}.   
Recall the admissible coordinates in \eqref{phi bounds} provide a diffeomorphism $\phi$ from a neighborhood of $[0,1]^2\subset\tilde\A$ to a 
neighborhood of $\cB=\cT_0$ in $\A$, satisfying $\phi\circ T=f^{-q_n}\circ\phi$.    Consider  
\ary\label{D:h_0}
				\hat{\phi}:=\phi\circ\Delta:\cT'_0\to\cT_0
\eary
where $\Delta$ denotes the map $(x,y)\mapsto (q_{n+1}x,y)$.   The idea is that by composing with $\Delta$ we obtain something like an admissible coordinate, but with 
domain the ''slice'' $\cT'_0=[1-1/q_{n+1},1]\times[0,1]\subset\A$, rather than $[0,1]^2$.   
%


\begin{lemma}\label{L:h with bounds}
The diffeomorphism $h_0:=\hat{\phi}^{-1}:\cT_0\to\cT'_0$ satisfies the matching condition \eqref{D:h_0 match}.   
Hence expression \eqref{D:h} defines an element $h\in\diff^{\infty}(\A)$ with 
\ary\label{conjugates g} 
		h\circ g=s\circ h. 
\eary
Moreover, $d_{\diff^{r_0}(\A)}(h,\id)$ is bounded in terms of $q_{n+1}, M, r_0$. 
\end{lemma}

\begin{proof} 
From \eqref{D:h_0} $h_0:=\hat{\phi}^{-1}$ has constant Jacobian.   
A simple computation shows $\Delta\circ R_{-1/q_{n+1}}=T^{-1}\circ\Delta$ 
on $\tilde\A$.    As $\phi$ has a smooth extension to a small open neighborhood of $[0,1]^2$, on which $\phi\circ T=f^{-q_n}\circ\phi$, so  
$\hat{\phi}$ has a smooth extension to a small open neighborhood of $\cT'_0$, in $\A$, on which 
\ary\label{hatphi fqn}
			\hat{\phi}\circ R_{-1/q_{n+1}}= f^{q_n}\circ\hat{\phi}.
\eary
Indeed, $\hat{\phi}\circ R_{-1/q_{n+1}}=\phi\circ\Delta\circ R_{-1/q_{n+1}}=\phi\circ T^{-1}\circ\Delta=f^{q_n}\circ\phi\circ\Delta=f^{q_n}\circ\hat{\phi}$.  
From \eqref{hatphi fqn} it follows that the smooth extension $h_0:\hat{\cT}_0\to\hat{\cT'}_0$ 
 between neighborhoods of $\cT_0$ and $\cT'_0$ in $\A$, given by the extension of $\hat{\phi}^{-1}$, 
satisfies 
\ary\label{h_0 fqn}
			h_0=R_{1/q_{n+1}}\circ h_0\circ f^{q_n}
\eary
on $\cT_0\cap\cT_{-q_n}$, which is an open neighborhood of $\gamma$.  
We claim that the right hand side of \eqref{h_0 fqn} is equal to $s^{-q_n}\circ h_0\circ g^{q_n}$ on a possibly smaller open neighborhood of $\gamma$.   
Indeed, $R_{1/q_{n+1}}=(R_{p_{n+1}/q_{n+1}})^{-p_n}=s^{-q_n}$ by \eqref{relation pn qn}, and  $g^{q_n}=f^{q_n}$ on $\cT_0$ by \eqref{L:g 2}.   
Meanwhile on $\cT_{-q_n}$ we have $g^{q_n}=\sigma_L\circ f^{q_n}$, from \eqref{L:g 2}.    
However, in \eqref{sigmaL conditions} we chose $\sigma_L$ to be the identity near to the 
left side of its domain, so $g^{q_n}(x)=f^{q_n}(x)$ if $x\in \cT_{-q_n}$ is sufficiently close to the left side of $\cT_{-q_n}$.    
This proves the claim, and from \eqref{h_0 fqn}
\aryst
							h_0=s^{-q_n}\circ h_0\circ g^{q_n}
\earyst 
on an open neighborhood of $\gamma$ in $\cT_0\cap\cT_{-q_n}$, proving \eqref{D:h_0 match}.   

To prove the bounds on $d_{\diff^{r_0}(\A)}(h,\id)$ it suffices to bound $\|h\|_{\diff^{r_0}(\A)}$ 
as $\A$ is compact.    By \eqref{D:h} it suffices to have analogous bounds on the maps $g$, $s$ and $h_0$ separately.   
Proposition \ref{P:sigma small} and \eqref{E:M bounds} give the bounds on $g=\sigma\circ f$, while $s$ is an isometry.
The bounds on $h_0=\Delta^{-1}\circ\phi^{-1}$ follow from \eqref{phi bounds}, as $\Delta$ is linear and depends only on $q_{n+1}$.   
\end{proof}

\subsection{An approximation with rotation number $\a$}\label{SS:approx a} 
Let $\epsilon>0$.   By Corollary \ref{L:periodic approximations}, if $q_{n+2}$ is sufficiently large depending on $q_{n+1}$ and $(r_0,M,\epsilon)$, there exists 
$g\in\diff^\infty(\A)$ satisfying $g^{q_{n+1}}=\id_{\A}$ and 
\ary\label{E:dist g}
				d_{\diff^{r_0}(\A)}(g,f)<\epsilon/2.    
\eary
Let $h\in\diff^\infty(\A)$ be the diffeomorphism produced by Lemma \ref{L:h with bounds}.   Then 
$g=h^{-1}\circ R_{p_{n+1}/q_{n+1}}\circ h$ 
and $d_{\diff^{r_0}(\A)}(h,\id)$ is bounded in terms of $q_{n+1}, M, r_0$.   If $q_{n+2}$ is even larger then: 
%

\begin{lemma}\label{P: approx with rot a}
If $q_{n+2}$ is sufficiently large depending on $q_{n+1}$ and $(r_0,M,\epsilon)$,
\ary\label{E:approx with rot a}
				d_{\diff^{r_0-1}(\A)}(h^{-1} R_{\rho(f)} h, f) < \epsilon.
\eary
\end{lemma}
\begin{proof}
The map $\R\to \diff^{r_0-1}(\A)$ given by $\theta\mapsto h^{-1}\circ R_{\theta}\circ h$ is continuous and equals $g$ at $\theta=p_{n+1}/q_{n+1}$.  
So there exists $\eta>0$ so that 
		\ary\label{E:dist g 2}
						d_{\diff^{r_0-1}(\A)}({h}^{-1} R_{\theta}{h},g)<\epsilon/2
		\eary
		for all $|\theta-p_{n+1}/q_{n+1}|<\eta$.   Due to the compact inclusion $\diff^{r_0}(\A)\to \diff^{r_0-1}(\A)$ the map 
		$\R\times\diff^{r_0}(\A)\to \diff^{r_0-1}(\A)$, $(\theta,\hat{h})\mapsto \hat{h}^{-1}\circ R_{\theta}\circ\hat{h}$ is uniformly continuous on bounded 
		subsets.   
		Hence $\eta$ depends only on $\epsilon$ and $d_{\diff^{r_0}(\A)}(h,\id)$, i.e., on $\epsilon$, $q_{n+1}, M$ and $r_0$.     
		If $q_{n+2}$ is sufficiently large compared with $\eta$, then 
                 $|\alpha-p_{n+1}/q_{n+1}|=\frac{1}{q_{n+1}}|q_{n+1}\alpha-p_{n+1}|=\beta_{n+1}/q_{n+1}\leq 1/(q_{n+1}q_{n+2})<\eta$ and so we can 
                 apply \eqref{E:dist g 2} to $\theta=\a$.    Combining with \eqref{E:dist g} 
                 gives \eqref{E:approx with rot a}.  
\end{proof}	
	
\subsection{Area preservation}\label{SS:area}
	
To complete the proof of Theorem \ref{thm. open dense}  it remains to modify $h$ from the last Lemma to make it area preserving.   

\begin{prop}\label{P: area preserving}
Given $\epsilon>0$, if $q_{n+2}$ is sufficiently large depending on $q_{n+1}$ and $(r_0,M,\epsilon)$, there exists $h\in\diff(\A,\omega)$ so that 
\ary\label{E:area preserving}
				d_{\diff^{r_0-2}(\A)}(h^{-1} R_{\rho(f)} h, f) < \epsilon.
\eary
\end{prop}
\begin{proof}
	Let $\epsilon>0$.   By Proposition \ref{P: approx with rot a} there exists $h_1\in\diff^\infty(\A)$ so that 
\ary\label{E:h_1}
				d_{\diff^{r_0-1}(\A)}(h_1^{-1} R_{\a} h_1, f) < \epsilon/2
\eary
where $\a=\rho(f)$.   
	Although $h_1$ is not close to the identity, we will show that $h_1^*\omega$ is close to $\omega$ and then apply Dacorogna-Moser \cite{DacMos}. 
	
	By Lemma \ref{L:composition} there exists $\delta>0$ so that 
	 \ary\label{E:h_2}
    d_{\diff^{r_0 - 2}(\A)}(h_2^{-1}h_1^{-1} R_{\a} h_1h_2, h_1^{-1} R_{\a} h_1) < \epsilon/2
    \eary
    	for every $h_2\in\diff^\infty(\A)$ satisfying 
	$d_{\diff^{r_0 - 2}(\A)}(h_2, {\rm Id})<\delta.$   
	Consider now the Jacobian function 
	\aryst
	\lambda := \det(Dh_1)   \in C^\infty(\tilde \A).
	\earyst
	We wish to show that $\lambda$ converges up to $r_0$-many derivatives to the constant function $1$, as $q_{n+2}$ goes to infinity.    
	From 
	\ary\label{integral lambda}
	\int_{\A} \lambda\,\omega = \int_{\A} h_1^*\omega = \int_{\A} \omega = 1
	\eary 
	there exists a point in $\A$ where $\lambda=1$.   Therefore it remains to show the derivatives of $\lambda$ can be made small if $q_{n+2}$ is large.  
%
	From \eqref{D:h}, on each $\cT_j$, $0\leq j\leq q_{n+1}-1$, we have $h_1=s^j\circ h_0\circ g^{-j}$, 
	where $h_0=\hat{\phi}^{-1}$, as in Lemma \ref{L:h with bounds}, has constant Jacobian.    Since $s=R_{p_{n+1}/q_{n+1}}$ has 
	constant Jacobian and $g=\sigma\circ f$ where $f$ has constant Jacobian, we conclude that on each $\cT_j$
	\aryst
	\log \lambda =  \log \det (D\sigma^{-j}) + c_j
	\earyst
	for some constant $c_j$.    
	Therefore, if $\sigma$ is chosen to come from Proposition \ref{P:sigma small}, then for any given $\epsilon_1>0$ we can find $q_{n+2}$ sufficiently large, depending on $q_{n+1}$ and $(r_0,M,\epsilon_1)$, so that 
	\ary\label{estimate loglambda}
	\qquad \sup_{1 \leq r \leq r_0-1}\norm{D^r(\log\lambda)}_{C^0(\A)}  &\leq& \epsilon_1, 
	\eary 
	since $\A$ is covered by $\cT_0,\ldots,\cT_{q_{n+1}-1}$.
	As observed from \eqref{integral lambda}, there exists a point where $\lambda=1$, so from \eqref{estimate loglambda}
	we conclude that if $q_{n+2}$ is sufficiently large, then $\log\lambda$, and hence $\lambda$, 
	is bounded on $\A$, by some constant depending only on $M$, $r_0$ and $q_{n+1}$.   
	Hence, there exists $C>0$ depending only on $(r_0,M,q_{n+1})$, so that 
\aryst
\norm{\lambda -1}_{C^{0}(\A)} <C\|\log\lambda - \log(1)\|_{C^{0}(\A)}< 2C \norm{D(\log \lambda)}_{C^{0}(\A)}.
\earyst
Consequently by \eqref{estimate loglambda}, for any $\epsilon_2 > 0$, we have $\norm{\lambda -1}_{C^{0}(\A)} < \epsilon_2$, 
if $q_{n+2}$ is sufficiently large depending on $(r_0,M, \epsilon_2)$ and $q_{n+1}$.  Now, using that $D\lambda=\lambda\cdot D(\log\lambda)$ 
and derivatives thereof, and the estimates \eqref{estimate loglambda}, we conclude inductively on the number of derivatives, that for all $\epsilon_3>0$, 
$\norm{\lambda -1}_{C^{r_0-1}(\A)} < \epsilon_3$, 
if $q_{n+2}$ is sufficiently large depending on $(r_0,M, \epsilon_3)$ and $q_{n+1}$.  
By  (\cite[Theorem 1]{DacMos}), there exists $h_2 \in \diff^\infty( \A )$ such that $(h_2)^* (\lambda\omega) = \omega$.  
    Moreover, by \cite[Theorem 2 and Lemma 3]{DacMos}, we can choose $h_2$ with $d_{\diff^{r_0 - 2}(\A)}(h_2, {\rm Id})<\delta$ 
    provided $\norm{\lambda -1}_{C^{r_0-1}(\A)}$ is sufficiently small.   Hence, we can assume that $h_2$ also satisfies \eqref{E:h_2}, 
    provided $q_{n+2}$ is sufficiently large depending on $(r_0,M, \epsilon)$ and $q_{n+1}$. 
    Now set $h:=h_1\circ h_2$.    Then combining \eqref{E:h_2} with \eqref{E:h_1} gives \eqref{E:area preserving}, and 
     $h^*\omega=(h_2)^*(h_1)^*\omega=(h_2)^*(\lambda\omega)=\omega$.  
\end{proof}

	
	This completes the proof of Theorem \ref{thm. open dense}.  
\end{proof}

\end{document}